\documentclass[12pt]{amsart}
\usepackage[margin=0.9in]{geometry}
\usepackage{amscd, amssymb, amsmath, wasysym}
\usepackage{graphicx}
\usepackage{amsfonts}
\usepackage{mathrsfs}    
\usepackage{mathtools}
\usepackage{eucal}     
\usepackage{latexsym}   
\usepackage{verbatim}   
\usepackage[all]{xy}   
\usepackage[dvipsnames]{xcolor}
\usepackage{bookmark}

\usepackage{hyperref}
 \hypersetup{
     colorlinks=true,
     linkcolor=NavyBlue,
     filecolor=NavyBlue,
     citecolor = TealBlue,
     urlcolor=magenta,
  }

\makeatletter

\newcounter{thmcounter}

\numberwithin{thmcounter}{section}
\numberwithin{equation}{section}

\newtheorem{theorem}[thmcounter]{Theorem}
\newtheorem{proposition}[thmcounter]{Proposition}
\newtheorem{lemma}[thmcounter]{Lemma}
\newtheorem{corollary}[thmcounter]{Corollary}

\theoremstyle{definition}

\newtheorem{remark}[thmcounter]{Remark}

\newtheoremstyle{claim}{9pt}{3pt}{}{\parindent}{\bf}{.}{1em}{}

\theoremstyle{claim}

\newenvironment{namelist}[1]{%
\begin{list}{}
{
\settowidth{\labelwidth}{#1}%
\setlength{\labelsep}{0.3em}%
\setlength{\leftmargin}{\labelwidth}%
\addtolength{\leftmargin}{\labelsep}}}{%
\end{list}}

\newcommand{\nC}{\mathbf{C}}  

\newcommand{\nP}{\mathbf{P}}

\newcommand{\sO}{\mathscr{O}}

\DeclareMathOperator{\ch}{ch}                    
\DeclareMathOperator{\coker}{coker}              
\DeclareMathOperator{\gon}{gon}                  
\DeclareMathOperator{\ev}{ev}					 
\DeclareMathOperator{\maxirr}{\text{Irr}_{\text{max}}}            
\DeclareMathOperator{\Pic}{Pic}                  
\DeclareMathOperator{\pr}{pr} 
\DeclareMathOperator{\Supp}{Supp}                
               
\DeclareMathOperator{\Sym}{Sym}        
\DeclareMathOperator{\td}{td}                    

\DeclareMathOperator{\rank}{rank}                

\newcounter{rkcounter}             
\begin{document}

\title[Graded Betti numbers of general curves of large degree]{Graded Betti numbers of general curves of large degree}

\author{JeongDon Lee}
\address{Department of Mathematical Sciences, KAIST, 291 Daehak-ro, Yuseong-gu, Daejeon 34141, Republic of Korea}
\email{s1865915@kaist.ac.kr}

\author{Li Li}
\address{Center for Complex Geometry, Institute for Basic Science (IBS), 55 Expo-ro, Yuseong-gu, Daejeon 34126, Republic of Korea}
\email{lili@ibs.re.kr}

\author{Jinhyung Park}
\address{Department of Mathematical Sciences, KAIST, 291 Daehak-ro, Yuseong-gu, Daejeon 34141, Republic of Korea}
\email{parkjh13@kaist.ac.kr}

\date{\today}



\thanks{J.P.~was partially supported by the National Research Foundation (NRF) funded by the Korea government (MSIT) (RS-2026-25478877 and RS-2026-25506097).}

\begin{abstract} 
Let $C$ be a smooth projective complex curve of genus $g$ and gonality $k$, and $L$ be a very ample line bundle on $C$. When $L$ has sufficiently large degree, the vanishing and nonvanishing of the Koszul cohomology groups $K_{p,q}(C,L)$ have been determined previously, but the exact values of the graded Betti numbers $\kappa_{p,q}(C, L)$ remain largely unknown. In this paper, we give explicit closed formulas for all graded Betti numbers $\kappa_{p,q}(C, L)$ when the Brill--Noether locus $W_k^1(C)$ has the expected dimension and $H^1(C, L \otimes \omega_C^{-1})=0$. Consequently, we determine the complete Betti table for a general curve when $\deg L \geq 4g-3$ or when $\deg L \geq 3g-3$ and $L$ is general. We also explicitly compute the Boij--S\"{o}derberg coefficient of the section ring $R(C, L)$ governing asymptotic purity, and show eventual monotonicity of the remaining coefficients: they decrease for hyperelliptic curves and increase under a natural generic reducedness assumption on the relevant Brill--Noether loci. 
\end{abstract}

\maketitle

\section{Introduction}

\noindent Let $C$ be a smooth projective complex curve of genus $g$, and $L$ be a very ample line bundle on $C$ with $d:= \deg L \geq 2g+1$. We set $r:=r(L)=h^0(C,L)-1$ and $k:=\gon(C)$, and write 
$$
\kappa_{p,q}(C,L):=\dim K_{p,q}(C,L),
$$
where $K_{p,q}(C,L)$ is the space of $p$-th syzygies of weight $q$ of the section ring $R(C,L)$ (see Subsection \ref{subsec:syzygies}). Note that $K_{p,0}(C,L) \neq 0$ if and only if $p=0$, and in this case, $\kappa_{0,0}(C,L)=1$. It is also known that $K_{p,q}(C,L)=0$ for $q \geq 3$. The only relevant Koszul cohomology groups are $K_{p,1}(C,L)$ and $K_{p,2}(C,L)$. Green's $(2g+1+p)$-theorem \cite[Theorem 4.a.1]{Green84} gives 
$$
K_{p,2}(C,L)=0~~\text{ for $0 \leq p \leq r-g-1$},
$$
while Green--Lazarsfeld \cite[Theorem 2]{GL88} showed the complementary nonvanishing result: if $H^0(C, L \otimes \omega_C^{-1}) \neq 0$, then 
$$
K_{p,2}(C,L) \neq 0~~\text{ for $r-g \leq p \leq r-1$}.
$$
On the other hand, Green--Lazarsfeld nonvanishing theorem \cite[Appendix]{Green84} yields that 
$$
K_{p,1}(C,L) \neq 0~~\text{ for $1 \leq p \leq r-k$}
$$
when $d \geq 2g+k-1$. Green--Lazarsfeld gonality conjecture \cite[Conjecture 3.7]{GL86} predicts that this range is sharp:
$$
K_{p,1}(C,L) = 0~~\text{ for $p \geq r-k+1$}.
$$
This conjecture was proved for $d \gg 0$ by Ein--Lazarsfeld \cite[Theorem A]{EL15} for general curves with $d \geq 2g+k-1$ by Farkas--Kemeny \cite[Theorem 0.1]{FK19}, and finally for arbitrary curves in sharp effective form by Niu--Park \cite[Theorem 1.1]{NP24+}. In particular, when $d \geq 3g-2$, the vanishing pattern of the Koszul cohomology groups $K_{p,1}(C,L)$ and $K_{p,2}(C,L)$ is completely understood.

\medskip

The remaining natural problem is to compute the exact values of the nonzero graded Betti numbers $\kappa_{p,1}(C,L)$ and $\kappa_{p,2}(C,L)$. This problem has attracted considerable attention in the recent development of asymptotic syzygies of algebraic varieties \cite{EL12, Park22, Park25} (it was posed explicitly in \cite[Problem 7.3]{EL12}). The difference $\kappa_{p,1}(C,L)-\kappa_{p-1,2}(C,L)$ is determined in terms of $d$ and $g$ (see (\ref{eq:k_1-k_2})), so one of the graded Betti numbers can be computed once the other vanishes. If $C$ is rational or elliptic, then all the graded Betti numbers are determined. We henceforth assume that $g \geq 2$. The difficulty begins in the overlap range $K_{p,1}(C,L) \neq 0$ and $K_{p-1,2}(C,L) \neq 0$. Ein--Lazarsfeld \cite[Theorem C]{EL15} proved that for $d=\deg L \gg 0$, both the graded Betti numbers $\kappa_{p,1}(C,L)$ and $\kappa_{p-1,2}(C,L)$ in this range are polynomials in $d$ whose degrees reflect the geometry of $C$. The leading coefficients and an effective range for this polynomial behavior were left open in general. When $C$ is hyperelliptic, one can compute all the graded Betti numbers using its special geometric properties (see Theorem \ref{thm:bettihyperelliptic}). Kemeny \cite[Theorem 1.1]{Kemeny20} computed the last nonzero weight-one graded Betti number for a general curve in each non-maximal gonality stratum (see also Remark \ref{rem:leadcoff}). For a curve of maximal gonality, however, no formula was known for the graded Betti numbers in the overlap range. This absence is notable since a general curve in the moduli space of curves of genus $g$ has maximal gonality $\lfloor (g+3)/2 \rfloor$, which is a very natural property from a Brill--Noether theoretic point of view. If $k=\lfloor (g+3)/2 \rfloor$ and the Brill--Noether locus 
$W_k^1(C)$ has the expected dimension $0$ when $g$ is even, then 
$$
C_p^1=\{\xi \in C_p \mid h^0(C, \sO_C(\xi)) \geq 2\}
$$
has the expected dimension $2p-g-1$ for every $k \leq p \leq g-1$. This property is precisely what enables us to compute the graded Betti numbers in the overlap range. Our first main theorem gives exact closed formulas for them. 

\begin{theorem}\label{thm:main1}
Assume that $C$ has maximal gonality $k=\lfloor (g+3)/2 \rfloor$, that $\dim W_k^1(C)=0$ if $g$ is even, and that  $H^1(C, L \otimes \omega_C^{-1})=0$. For $k \leq p \leq g-1$, we have 
$$
\kappa_{r-p, 1}(C, L)= \sum_{i=g-p}^p {g \choose i} \left( (i-g+p) {d- 2g \choose p-i} - {d - 2g \choose p-i-1} \right)
$$
and
$$
\kappa_{r-p-1,2}(C,L)=\sum_{i=0}^{g-p-1} {g \choose i} \left( (g-p-i) {d-2g \choose p-i} + {d-2g \choose p-i-1} \right).
$$
\end{theorem}

Theorem \ref{thm:main1} determines all previously undetermined graded Betti numbers, and combined with the known vanishing results, it therefore determines the complete Betti table of $R(C,L)$. The cohomology condition in the theorem is automatic either when $d \geq 4g-3$ or when $L$ is general and $d \geq 3g-3$. Notice that if $C$ is a general curve, then $C$ has maximal gonality and $W_k^1(C)$ has the expected dimension. Hence the theorem determines the complete Betti table of $R(C,L)$ when $C$ is a general curve and $L$ is a general line bundle with $d \geq 3g-3$. 

\medskip

The formulas in Theorem \ref{thm:main1} provide a sharp extremal principle. Curves satisfying the hypotheses of Theorem 1.1 realize the lower bounds for $\kappa_{r-p, 1}(C, L)$ and $\kappa_{r-p-1,2}(C,L)$, and hyperelliptic curves attain the upper bounds. 

\begin{corollary}\label{cor:inequality}
Assume that $d=\deg L \gg 0$. Then
$$
\sum_{i=g-p}^p {g \choose i} \left( (i-g+p) {d- 2g \choose p-i} - {d - 2g \choose p-i-1} \right)
\leq \kappa_{r-p, 1}(C, L) \leq
(d-g-p){d-g-1 \choose p-2}
$$
and
$$
\sum_{i=0}^{g-p-1} {g \choose i} \left( (g-p-i) {d-2g \choose p-i} + {d-2g \choose p-i-1} \right) 
\leq \kappa_{r-p-1,2}(C,L) \leq
(g-p){d-g-1 \choose p}.
$$
for $k \leq p \leq g-1$. 
If the equality on the left holds for all $p$, then $C$ has maximal gonality and $\dim W_k^1(C)=0$ when $g$ is even. 
If the equality on the right holds for some $p$, then $C$ is hyperelliptic, and consequently, the equality on the right holds for all $p$.
\end{corollary}

We briefly describe our approach to proving Theorem \ref{thm:main1}. For $k \leq p \leq g-1$, consider the evaluation map 
$$
\operatorname{ev}_{p, \omega_C} \colon H^0(C, \omega_C) \otimes \mathscr{O}_{C_p} \longrightarrow E_{p, \omega_C},
$$
of the tautological bundle $E_{p, \omega_C}$ on the symmetric product $C_p$, and let $\mathscr{F}_{p, \omega_C}$ be its kernel. By the methods developed by Ein--Lazarsfeld \cite{EL15} and Voisin \cite{Voisin02, Voisin05}, we have
$$
\kappa_{r-p-1, 2}(C, L) =  h^0(C_p, \mathscr{F}_{p, \omega_C} \otimes N_{p,L}). 
$$
By a well-known computation in Brill--Noether theory, the degeneracy locus of $\operatorname{ev}_{p, \omega_C}$ is exactly $C_p^1$. Assume that $C_p^1$ has the expected dimension. Then the kernel $\mathscr{F}_{p, \omega_C}$ of  $\operatorname{ev}_{p, \omega_C}$ is resolved by the Buchsbaum--Rim complex  $\operatorname{BR}_{p,\bullet}$. The condition $H^1(C, L \otimes \omega_C^{-1})=0$ allows us to get
$$
\kappa_{r-p-1, 2}(C, L) = \chi(\mathscr{F}_{p, \omega_C} \otimes N_{p,L})
$$
by higher cohomology vanishing for $\operatorname{BR}_{p,i} \otimes N_{p,L}$ (see Proposition \ref{prop:vanishing}).
We then reduce the problem to computing the alternating sum of the Euler characteristics $\chi(\operatorname{BR}_{p,i} \otimes N_{p,L})$. This can be done by Hirzebruch--Riemann--Roch theorem (see Proposition \ref{prop:eulercharacteristic}). Our method is not restricted to the case that $C$ has maximal gonality. More generally, if $W_m^1(C)$ has the expected dimension, then the formula in Theorem \ref{thm:main1} holds for $m \leq p \leq g-1$ (see Theorem \ref{thm:bettinumber}).

\medskip

At the boundary $p=g-1$, if $\dim C_{g-1}^1=g-3$ (equivalently, $C$ is not hyperelliptic by Martens theorem), then the Buchsbaum--Rim complex yields $\mathscr{F}_{g-1, \omega_C} = N_{g-1, \omega_C}^{-1}$. Thus this together with Theorem \ref{thm:bettihyperelliptic} gives the exact value of the first potentially nonzero graded Betti number of weight two under the weaker assumption $d \geq 2g+1$. The next theorem refines the aforementioned nonvanishing result of Green--Lazarsfeld \cite[Theorem 2]{GL88} by determining its exact value. 

\begin{theorem}\label{thm:main2}
Assume that $d=\deg L \geq 2g+1$. Then
$$
\kappa_{r-g, 2}(C, L) = \begin{cases}
\displaystyle {d - g - 1 \choose g-1} & \text{if $C$ is hyperelliptic} \\[10pt]
\displaystyle  {h^0(C, L \otimes \omega_C^{-1})+g-2 \choose g-1} & \text{if $C$ is not hyperelliptic}.
\end{cases}
$$
\end{theorem}

This theorem shows that the condition $H^1(C, L \otimes \omega_C^{-1})=0$ in Theorem \ref{thm:main1} is genuine rather than an artifact of the proof. Even the first potentially nonzero weight-two graded Betti number may not be given by the formula of Theorem \ref{thm:main1}. For example, if $C$ is a general curve of genus $4$ and $L=\omega_C^2$, then $\kappa_{5,1}(C,L) = 20$, whereas the formula in Theorem \ref{thm:main1} gives $10$. This example, in particular, shows that the assumption $\deg L \geq 2g+k+1$ as in \cite[Theorem 1.1]{Kemeny20} and \cite[Theorem 1.1]{NP24+} is not sufficient for Theorem \ref{thm:main1}.

\medskip

Our results and methods have several further consequences. Applying Theorem \ref{thm:main2} to a general curve section of a rational surface with effective anticanonical divisor, we can determine the first nonzero graded Betti number of weight 2 (see Corollary \ref{cor:rationalsurface}). In the toric setting, related problems have been studied from a combinatorial perspective (see e.g., \cite{CCDL19, Lemmens18}). Our approach is entirely different and provides stronger results.
On the other hand, under some additional conditions, we also compute the next weight-two graded Betti number $\kappa_{r-g+1,2}(C,L)$ explicitly (see Theorem \ref{thm:kappa_{r-g+1,2}}), and consequently, we recover Agostini's result on Green--Lazarsfeld secant conjecture \cite[Main Theorem]{Agostini24} (see Corollary \ref{cor:Agostini}). 

\medskip

Finally, we turn to the Boij--S\"{o}derberg decomposition of $R(C, L)$ (see Subsection \ref{subsec:BS}). By Eisenbud--Schreyer \cite{ES09}, the Betti table of $R(C,L)$ is a linear combination of pure Betti tables, and the coefficients, called the Boij--S\"{o}derberg coefficients, are positive rational numbers. It is natural to ask how the Boij--S\"{o}derberg coefficients of $R(C, L)$ behave asymptotically as $d$ grows (cf. \cite[Problem 7.4]{EL12}). Erman's asymptotic purity result  \cite{Erman15} says that the distinguished Boij--S\"{o}derberg coefficient $c_g$ tends to $1$ and other Boij--S\"{o}derberg coefficients $c_{k-1}, \ldots, c_{g-1}$ tend to $0$ as $d \to \infty$ (cf. when $g=0$ or $1$, we have $c_g=1$ and all the other coefficients are exactly $0$).
Theorem \ref{thm:main2} computes the distinguished coefficient $c_g$ exactly (see Corollary \ref{cor:compBScoeff}). The remaining problem is to understand the other coefficients $c_{k-1}, \ldots, c_{g-1}$. The two extremal geometries of algebraic curves produce opposite monotonicity for the Boij--S\"{o}derberg coefficients: $c_{k-1}, \ldots, c_{g-1}$ decrease for hyperelliptic curves and increase for general curves. More generally, we show in the following theorem that the increasing behavior holds for a nonhyperelliptic curve when every maximal component of the Brill--Noether locus $W_p^1(C)$ is generically reduced for $k-1 \leq p \leq g+1$. Note that this Brill--Noether condition holds when $C$ is a general curve of prescribed gonality (see Remark \ref{rem:HBN}).

\begin{theorem}\label{thm:main3}
Assume that $d = \deg L \gg 0$. Let $c_{k-1}, c_k, \ldots, c_g$ be the Boij--S\"{o}derberg coefficients of $R(C,L)$. Then the following hold:
\begin{enumerate}
\item We have
$$
c_g=\begin{cases}
\displaystyle \frac{(d-2g+1)(d-2g+2)}{(d-g)(d-g+1)} & \text{if $C$ is hyperelliptic} \\[10pt]
\displaystyle \frac{(d-2g+1)!(d-2g+2)!}{(d-3g+2)!(d-g+1)!} & \text{if $C$ is not hyperelliptic}.
\end{cases}
$$
\item Assume that $C$ is hyperelliptic. Then 
$$
c_1 > c_2 > \cdots > c_{g-1}.
$$
\item Assume that $C$ is not hyperelliptic and every maximal irreducible component of the Brill--Noether locus $W_{p}^1(C)$ is generically reduced for each $k-1 \leq p \leq g+1$. Then
$$
c_{k-1} < c_k < \cdots < c_{g-1}.
$$
\end{enumerate}
\end{theorem}

To prove Theorem \ref{thm:main3}, we measure the contribution of each maximal component of $W_{p}^1(C)$ to $\kappa_{r-p,1}(C,L)$, and compare these contributions as $p$ varies. The Brill--Noether loci governing the graded Betti numbers also control the Boij--S\"{o}derberg coefficients. 

\medskip

This paper is organized as follows:
In Section \ref{sec:prelim}, we collect the necessary background and fix notation. We prove Theorems \ref{thm:main1} and \ref{thm:main2} in Section \ref{sec:betti} and Theorem \ref{thm:main3} in Section \ref{sec:BS}.

\subsection*{Acknowledgements}
We are grateful to Gavril Farkas and Frank-Olaf Schreyer for insightful and inspiring discussions. We used ChatGPT (GPT-5.6 Sol Pro, OpenAI) to assist in managing the combinatorial computations arising after applying Hirzebruch--Riemann--Roch theorem in the proof of Theorem \ref{thm:main1}.

\section{Preliminaries}\label{sec:prelim}
\noindent In this section, we review syzygies of algebraic varieties, symmetric products of algebraic curves, Brill--Noether theory, and Boij--S\"{o}derberg theory. Throughout the paper, $C$ is a smooth projective complex curve of genus $g \geq 2$ and gonality $k$, and $L$ is a very ample line bundle on $C$. We set $d:=\deg L$ and $r:=r(L)=h^0(C, L)-1$. When $L$ is nonspecial, i.e., $H^1(C, L)=0$, we have $r=d-g$. 

\subsection{Koszul cohomology}\label{subsec:syzygies}
Let $X$ be a smooth projective complex variety, $B$ be a line bundle on $X$, and $H$ be a globally generated line bundle on $X$. Then the section module 
$$
R(X, B; H):= \bigoplus_{m \geq 0} H^0(X, B \otimes H^m)
$$
over $S:=\Sym H^0(X, H)$ admits a minimal graded free resolution
$$
0 \longleftarrow R(X, B; H) \longleftarrow E_0 \longleftarrow E_1 \longleftarrow \cdots \longleftarrow E_p \longleftarrow E_{p+1} \longleftarrow \cdots,
$$
where 
$$
E_p = \bigoplus_{q \geq 0} K_{p,q}(X, B; H) \otimes_{\nC} S(-p-q).
$$
The Koszul cohomology group $K_{p,q}(X, B; H)$ can be regarded as the space of $p$-th syzygies of weight $q$ of $R(X, B; H)$. We write
$$
\kappa_{p,q}(X, B; H):= \dim K_{p,q}(X, B; H)
$$
for the graded Betti number. We omit $B$ when $B= \sO_X$. It is well-known that 
$$
K_{p,q}(X, B;H) = \coker\big( \wedge^{p+1} H^0(X, H) \otimes H^0(X, B \otimes H^{q-1}) \longrightarrow H^0(X, \wedge^p M_H \otimes B \otimes H^q)\big),
$$
where $M_H$ is the kernel of the evaluation map $H^0(X, H) \otimes \sO_X \to H$ (see  \cite[Section 2]{AN10} or \cite[Section 3]{EL12}). 
Now, suppose that $X=C$ is a curve and $L$ is nonspecial with $H^0(C, \omega_C \otimes L^{-1})=0$. Notice that $\rank M_L = r$ and $\det M_L = L^{-1}$. Then 
$$
K_{p,2}(C, L) = H^1(C, \wedge^{p+1} M_L \otimes L) = H^0(C, \wedge^{r-p-1}M_L  \otimes \omega_C)^{\vee} = K_{r-p-1, 0}(C, \omega_C; L)^{\vee}~\text{for $0 \le p \le r-1$}.
$$
By the duality theorem (\cite[Theorem 2.c.6]{Green84}), we have 
$$
K_{p,1}(C, L) = K_{r-p-1, 1}(C, \omega_C; L)^{\vee}~~\text{ for $p \geq 0$}.
$$
By Riemann--Roch theorem, we find
\begin{align}\label{eq:k_1-k_2}
\begin{split}
\kappa_{p+1,1}(C, L) - \kappa_{p,2}(C, L) 
&=\chi(\wedge^{p+1} M_L \otimes L) - \dim \wedge^{p+2}H^0(C, L)\\
&= (p+1) { d - g \choose p+1} \left( \frac{d+1-g}{p+2} - \frac{d}{d-g} \right)~~\text{ for $0 \le p \le r-1$}.
\end{split}
\end{align}
When the gonality conjecture holds, together with Green's $(2g+1+p)$-theorem, we see that the Betti table of $R(C,L)$ is as follows:
\begin{center}
\texttt{ \begin{tabular}{c|cccccccccc}
         & $0$ & $1$ & $\cdots$ & $r-g-1$ & $r-g$ & $\cdots$ & $r-k$ & $r-k+1$ & $\cdots$ & $r-1$ \\ \hline
    $0$    & $1$ & $-$ & $\cdots$ & $-$ & $-$ & $\cdots$ & $-$ & $-$ & $\cdots$ & $-$ \\
    $1$   & $-$ & $*$ & $\cdots$ & $*$ & $*$ & $\cdots$ & $*$ & $-$ & $\cdots$ & $-$ \\
    $2$   & $-$ & $-$ & $\cdots$ & $-$ & $*$ & $\cdots$ & $*$ & $*$ & $\cdots$ & $*$ \\
\end{tabular}}
\end{center}
Here $-$ indicates a zero entry, and $*$ indicates a possibly nonzero entry. By (\ref{eq:k_1-k_2}), we can determine $\kappa_{p+1,1}(C,L)$ for $0 \leq p \leq r-g-1$ and $\kappa_{p,2}(C,L)$ for $r-k \leq p \leq r-1$. Therefore, the only undetermined graded Betti numbers are $\kappa_{r-p,1}(C,L) = \kappa_{p-1,1}(C, \omega_C; L)$ and $\kappa_{r-p-1,2}(C,L) = \kappa_{p,0}(C, \omega_C; L)$ for $k \leq p \leq g-1$. We refer to \cite{AN10, EL12, EL26} for further background.

\subsection{Symmetric products of algebraic curves}\label{subsec:symmetric}
For an integer $p \geq 1$, let $C_p$ be the $p$-th symmetric product of $C$, which is a smooth projective complex variety of dimension $p$. For $p,m \geq 1$, there is an addition map 
$$
\sigma_{p,m} \colon C_p \times C_m \longrightarrow C_{p+m},~~(\xi, \eta) \longmapsto \xi + \eta,
$$
which is a finite morphism. Consider the incidence divisor
$$
D_{p,m}:=\{ (\xi, \eta) \in C_p \times C_m \mid \Supp(\xi) \cap \Supp(\eta) \neq \emptyset\} \subseteq C_p \times C_m.
$$
When $m=1$, we have $C_{p-1} \times C \cong D_{p,1} \subseteq C_{p} \times C$, and $\sigma_{p-1,1} = \pr_1|_{D_{p,1}}$, where $\pr_1 \colon C_p \times C_1 \to C_p$ is the first projection. For a line bundle $B$ on $C$, define
$$
E_{p,B}:=\sigma_{p-1,1,*} (\sO_{C_{p-1}} \boxtimes B)~~\text{ and }~~N_{p,B}:=\det E_{p, B}.
$$
Note that $E_{p,B}$ is a vector bundle of rank $p$ and $H^0(C_p, E_{p,B}) = H^0(C, B)$. It is well-known that $E_{p, \omega_C}=\Omega_{C_p}^1$ is the cotangent bundle and $N_{p, \omega_C}=\omega_{C_p}$ is the canonical line bundle.
There is a divisor $\delta_p$ on $C_p$ with $\sO_{C_p}(-\delta_p) = N_{p, \sO_C}$. By \cite[Lemma 2.3]{NP24+} (see also \cite[Lemma 3.1]{Agostini24}), 
$$
\sigma_{p,m}^* \sO_{C_{p+m}}(-\delta_{p+m}) = (\sO_{C_p}(-\delta_p) \boxtimes \sO_{C_m}(-\delta_m))(-D_{p,m}).
$$
In particular, $\sigma_{p-1,1}^*\sO_{C_p}(-\delta_p) = (\sO_{C_{p-1}}(-\delta_{p-1}) \boxtimes \sO_C)(-D_{p-1,1})$. We set 
$$
S_{p,B}:=N_{p,B}(\delta_p). 
$$
Note that $S_{p,B} \otimes S_{p, B'} = S_{p, B \otimes B'}$ for a line bundle $B'$ on $C$ and $\sigma_{p,m}^* S_{p+m, B} = S_{p, B} \boxtimes S_{m, B}$. Now, we define two divisor classes on $C_p$ following \cite[Section VII.5]{ACGH85}. For a point $z \in C$, consider the effective divisor 
$$
X_{p,z}:=z + C_{p-1}=\{ \xi \in C_p \mid z \in \Supp(\xi)\} \subseteq C_p.
$$
Note that $\sO_{C_p}(X_{p,z}) = S_{p, \sO_C(z)}$. Define $x_p:=[X_{p,z}] \in H^2(C_p, \mathbf{Z})$ and $\theta_p \in H^2(C_p, \mathbf{Z})$ to be the divisor class of the pull-back of the theta divisor of $\Pic_p(C)$ via the Abel--Jacobi map $u_p \colon C_p \to \Pic_p(C)$. We have
\begin{equation}\label{eq:xtheta}
x_p^{p-j} \theta_p^j = \frac{g!}{(g-j)!} ~~\text{ for $0 \leq j \leq \min\{p, g\}$}.
\end{equation}
Next, consider the evaluation map
$$
\ev_{p, B} \colon H^0(C, B) \otimes \sO_{C_p} \longrightarrow E_{p, B}, 
$$
and let $\mathscr{F}_{p,B}$ and $\mathscr{G}_{p,B}$ be its kernel and cokernel, respectively. Note that $B$ is $(p-1)$-very ample if and only if $\ev_{p, B}$ is surjective. In this case, the kernel $M_{p, B}:=\mathscr{F}_{p,B}$ of $\ev_{p,B}$ is a vector bundle. When $p=1$, this agrees with the usual kernel bundle $M_B$. 
If $B$ is $(p-1)$-very ample, then \cite[Lemma 2.4]{NP24+} gives
$$
R^i \pr_{1,*} (\sO_{C_p} \boxtimes N_{m,B})(-D_{p,m}) = \wedge^{m-i} M_{p,B} \otimes S^i H^1(C, B)~~\text{for $i \geq 0$}, 
$$
where $\pr_1 \colon C_p \times C_m \to C_p$ is the first projection. Assume that $H^0(C, B \otimes L^{-1})=0$ and $H^1(C, L)=0$. Then 
$$
K_{p,0}(C, B; L) = H^0(C, \wedge^p M_L \otimes B) 
= H^0(C_p, \mathscr{F}_{p, B} \otimes N_{p, L}).
$$
When $B$ is $(p-1)$-very ample, we have
$$
K_{p-1,1}(C, B; L) = H^1(C_p, M_{p,B} \otimes N_{p,L})
$$
(see \cite[Lemma 1.1]{EL15}). 
When $d=\deg L \gg 0$ so that the line bundle $N_{p,L}$ is sufficiently positive, we also have
$$
K_{p-1,1}(C, B; L) = H^0(C_p, \mathscr{G}_{p,B} \otimes N_{p,L}).
$$
In this case, we see that $\kappa_{p,0}(C, B; L)$ and $\kappa_{p-1,1}(C, B; L)$ are polynomials in $d$ (see \cite[Theorem C]{EL15}). 
We refer to \cite{Agostini24, ACGH85, EL15, ENP20, NP24+} for further details.

\subsection{Brill--Noether loci}\label{subsec:BN}
For integers $p \geq s \geq 1$, let
$$
C_p^s:=\{ \xi \in C_p \mid \rank (du_p)^{\vee}(\xi) \leq p-s\}
$$
be the degeneracy locus with the determinantal scheme structure, where
$$
(du_p)^{\vee} \colon u_p^* \Omega_{\Pic_p(C)}^1 \longrightarrow \Omega_{C_p}^1
$$
is the dual of the differential map of the Abel--Jacobi map $u_p$, which can be naturally identified with the evaluation map $\ev_{p, \omega_C} \colon H^0(C, \omega_C) \otimes \sO_{C_p} \to E_{p, \omega_C}$. As a set,
$$
C_p^s=\{ \xi \in C_p \mid h^0(C, \sO_C(\xi)) \geq s+1\}.
$$
One can also give a determinantal scheme structure on the Brill--Noether locus
$$
W_p^s(C):=\{ M \in \Pic_p(C) \mid r(M)=h^0(C, M)-1 \geq s\}
$$
(see \cite[Section IV.3]{ACGH85}). Then $C_p^s=u_p^{-1}(W_p^s(C))$ is the scheme-theoretic inverse image by \cite[Proposition IV.3.4]{ACGH85}. When $g-p+s \geq 0$, \cite[Lemma IV.1.7]{ACGH85} states that no component of $C_p^s$ is contained in $C_p^{s+1}$. Thus 
$$
\dim C_p^s = \dim W_p^s(C)+s,
$$
and the restricted Abel--Jacobi map $u_p \colon C_p^s \to W_p^s(C)$ is a generically $\nP^s$-fibration. In this paper, we focus ONLY on the case $s=1$. Note that $W_p^1(C)=\emptyset$ for $1 \leq p \leq k-1$. For $k \leq p \leq g$, we have $\dim W_{p+1}^1(C) \geq \dim W_p^1(C) +1$. But \cite[Theorem 1]{FHL84} implies that $\dim W_{p+1}^1(C) \leq \dim W_p^1(C) + 2$, so $\dim W_{p+1}^1(C) - \dim W_p^1(C) = 1$ or $2$. The expected dimension of $W_p^1(C)$ is $2p-g-2$.
In any event, $W_{g+1}^1(C)$ and $W_g^1(C)$ have expected dimensions, i.e., $\dim W_{g+1}^1(C) = g$ and $\dim W_g^1(C) = g-2$. Note that $W_p^1(C)$ has the expected dimension $2p-g-2$ for every $k \leq p \leq g-1$ if and only if $C$ has maximal gonality $k=\lfloor (g+3)/2 \rfloor$ and $\dim W_k^1(C)=0$ when $g$ is even. 
As $\dim W_p^1(C) \leq p-2$, Martens theorem (see e.g., \cite[Theorem IV.5.1]{ACGH85}) shows that $\dim W_p^1(C) = p-2$ for some $k \leq p \leq g-1$ if and only if $C$ is a hyperelliptic curve. In this case, $\dim W_p^1(C) = p-2$ for all $k \leq p \leq g-1$. 
We refer to \cite{ACGH85} for more details on the classical aspects of Brill--Noether theory.

\subsection{Boij--S\"{o}derberg coefficients}\label{subsec:BS}
Eisenbud--Schreyer \cite{ES09} proved that the Betti table of a finitely generated graded Cohen--Macaulay module over a polynomial ring is a positive rational combination of pure Betti tables, which is called the Boij--S\"{o}derberg decomposition. For an increasing integer sequence $\textbf{d}=(d_0, d_1, \ldots, d_{e-1}, d_{e})$, let $\pi(\textbf{d})$ be the normalized pure Betti table whose nonzero entries are $\kappa_{0,d_0}(\pi(\textbf{d}))=1$ and
$$
\kappa_{p,d_p-p}(\pi(\textbf{d})):=\prod_{\substack{k \geq 1 \\ k \neq p}} \frac{d_k-d_0}{|d_k - d_p|}~~\text{ for $1 \leq p \leq e$}.
$$
Assume that $d=\deg L$ is large enough so that the gonality conjecture holds by \cite[Theorem 1.1]{NP24+}. Let $\pi(R(C,L))$ be the Betti table of the section ring $R(C,L)$. The Boij--S\"{o}derberg decomposition of $R(C,L)$ can be written as
$$
\pi(R(C,L)) = \sum_{i=k-1}^{g} c_i \pi(\textbf{d}^i),
$$
where $c_i$ are positive rational numbers and
$$
\textbf{d}^i:=(0,2,\ldots, r-i, r-i+2, \ldots, r+1)~~\text{ for $k-1 \leq i \leq g$}.
$$
Here $c_{k-1}, c_k, \ldots, c_g$ are the Boij--S\"{o}derberg coefficients of $R(C,L)$. From $\kappa_{0,0}(C,L)=1$, we get $\sum_{i=k-1}^{g} c_i = 1$. Since
$$
\kappa_{r-p,1}(\pi(\textbf{d}^i)) = \frac{(p-i)(r-p)}{r-i+1} {r+1 \choose p}~~\text{ for $i+1 \leq p \leq r-1$},
$$
we obtain
\begin{equation}\label{eq:BSforcurves}
\kappa_{r-p,1}(C,L) = \sum_{i=k-1}^{p-1} c_i \frac{(p-i)(r-p)}{r-i+1} {r+1 \choose p}~~\text{ for $k \leq p \leq g+1$}.
\end{equation}
We refer to \cite[Lecture 2]{EL26} for an overview of Boij--S\"{o}derberg theory.

\section{Graded Betti numbers of algebraic curves}\label{sec:betti}
\noindent In this section, we prove Theorems \ref{thm:main1} and \ref{thm:main2}, and present several applications. For Theorem \ref{thm:main1}, we follow the strategy sketched in the introduction. We start by showing a necessary cohomology vanishing result on $C_p$. 

\begin{proposition}\label{prop:vanishing}
Let $A, B$ be line bundles on $C$ such that $H^1(C, A)=0$. For any integer $p \geq 1$, we have
$$
H^i(C_p, S^j E_{p, B}^{\vee} \otimes S_{p, A})=0~~\text{ for $i>0$ and $0 \leq j \leq i-1$}. 
$$
\end{proposition}

\begin{proof}
We proceed by induction on $p$. If $p=1$, then the assertion is $H^1(C, A)=0$, which is our assumption. Consider the case that $p \geq 2$. When $j=0$, we have 
$$
H^i(C_p, S_{p, A})= S^{p-i} H^0(C, A) \otimes \wedge^i H^1(C, A) = 0~~\text{ for $i>0$},
$$
which proves the assertion. We henceforth assume that $1 \leq j \leq i-1$ (and hence $i \geq 2$). Consider the finite morphism $\sigma_{p-1,1} \colon C_{p-1} \times C \to C_p$ given by $(\xi, x) \mapsto \xi+x$. As $\sO_{C_p}$ is a direct summand of $\sigma_{p-1,1,*} \sO_{C_{p-1} \times C}$, it is enough to show that
$$
H^i(C_{p-1} \times C, \sigma_{p-1,1}^* S^j E_{p,B}^{\vee} \otimes (S_{p-1, A} \boxtimes A))=0~~\text{ for $i \geq 2$ and $1 \leq j \leq i-1$}. 
$$
We have a short exact sequence
$$
0 \longrightarrow E_{p-1, B}^{\vee} \boxtimes \sO_C \longrightarrow \sigma_{p-1,1}^* E_{p, B}^{\vee} \longrightarrow (\sO_{C_{p-1}} \boxtimes B ^ \vee )(D_{p-1,1}) \longrightarrow 0. 
$$
It suffices to check that
$$
H^i(C_{p-1} \times C, (S^{\ell} E_{p-1, B}^{\vee} \otimes S_{p-1, A} \boxtimes B^{-(j - \ell)} \otimes A)((j-\ell)D_{p-1,1}))=0
$$
for $i \geq 2, ~1 \leq j \leq i-1, ~0 \leq \ell \leq j$. When $\ell=j$, we have
\begin{align*}
& H^i(C_{p-1} \times C, S^j E_{p-1,B}^{\vee} \otimes S_{p-1, A} \boxtimes A)\\
& = H^i(C_{p-1}, S^j E_{p-1, B}^{\vee} \otimes S_{p-1, A}) \otimes H^0(C, A) = 0
\end{align*}
for $i  \geq 2$ and $ 1\leq j \leq i-1$ by K\"{u}nneth formula and induction hypothesis. For $0 \leq \ell \leq j-1$, considering the Leray spectral sequence for the second projection $\pr_2 \colon C_{p-1} \times C \to C$, we reduce the problem to 
$$
H^i(C_{p-1}, S^{\ell} E_{p-1, B}^{\vee} \otimes S_{p-1, A((j-\ell)x)})=H^{i-1}(C_{p-1}, S^{\ell} E_{p-1,B}^{\vee} \otimes S_{p-1, A((j-\ell)x)})=0
$$
for $i \geq 2$ and $1 \leq j \leq i-1$. Since $0 \leq \ell \leq i-2$ and $H^1(C, A((j-\ell)x))=0$, these cohomology vanishings follow from the induction hypothesis.
\end{proof}

This proposition may be viewed as a counterpart to Rathmann's theorem \cite[Theorem 3.1]{Rathmann16+}, which essentially establishes cohomology vanishing for exterior powers of $M_{p+1,B}$ and implies that  the gonality conjecture holds when $\deg L \geq 4g-3$ (see \cite[Theorems 1.1 and 1.2]{Rathmann16+}). We can recover this result using our proposition. 

\begin{corollary}[Rathmann]\label{cor:Rathmann}
Let $B$ be a $p$-very ample line bundle on $C$ with $H^1(C, L)=H^1(C, L \otimes B^{-1})=0$. Then 
$$
H^i(C_{p+1}, \wedge^j M_{p+1,B} \otimes N_{p+1,L})=0~~\text{ for $i>0$ and $j \geq 0$}.
$$
In particular, $K_{p,1}(C,B;L)=0$. 
\end{corollary}

\begin{proof}
We have a short exact sequence
$$
0 \longrightarrow M_{p+1,B} \longrightarrow H^0(C, B) \otimes \sO_{C_{p+1}} \xrightarrow{~\ev_{p+1,B}~} E_{p+1,B} \longrightarrow 0
$$
of vector bundles on $C_{p+1}$. Note that $\wedge^{r_B-j} M_{p+1,B}^{\vee} = \wedge^j M_{p+1,B} \otimes N_{p+1,B}$, where $r_B:=\rank M_{p+1,B}$. Taking the $(r_B-j)$-th exterior power of the dual of the above short exact sequence and tensoring with  $N_{p+1,B}^{-1}$, we obtain the following long exact sequence:
\begin{align*}
& \cdots \longrightarrow \wedge^{r_B-j-2} H^0(C, B)^{\vee} \otimes S^2 E_{p+1,B}^{\vee} \otimes N_{p+1,B} ^ {-1} \longrightarrow \wedge^{r_B-j-1} H^0(C, B)^{\vee} \otimes E_{p+1,B}^{\vee} \otimes N_{p+1,B}^{-1} \quad \quad \quad \\
& \quad \quad \quad \quad \quad \quad \quad \quad \quad \quad \quad \quad \quad \quad \quad \quad \quad \quad \longrightarrow
\wedge^{r_B-j} H^0(C, B)^{\vee} \otimes N_{p+1,B}^{-1}\longrightarrow  \wedge^j M_{p+1,B} \longrightarrow 0.
\end{align*}
Note that $N_{p+1,L} \otimes N_{p+1,B}^{-1} = S_{p+1, A}$, where $A:=L \otimes B^{-1}$.
Chasing through the long exact sequence (see \cite[Proposition B.1.2]{Lazarsfeld04}), we see that the first assertion follows from 
$$
H^i(C_{p+1}, S^j E_{p+1,B}^{\vee} \otimes S_{p+1,A})=0~~\text{ for $i>0$ and $0 \leq j \leq i-1$},
$$
which holds by Proposition \ref{prop:vanishing}. The second assertion now follows from
\[
K_{p,1}(C, B;L) = H^1(C_{p+1}, M_{p+1,B} \otimes N_{p+1,L})=0. \qedhere
\]
\end{proof}

\begin{remark}
The argument used in the proof deducing the cohomology vanishings of exterior powers of $M_{p+1,B}$ from those of symmetric powers of $E_{p+1,B}^{\vee}$ was also employed in \cite{Agostini24} (see also Remark \ref{rem:Agostini}) and \cite{CLPS25+}. The long exact sequence in the proof for $j=1$ is actually the Buchsbaum--Rim complex for the evaluation map $\ev_{p+1,B}$. This also plays a crucial role in the proof of Theorem \ref{thm:main1}. 
\end{remark}

Having established the required cohomology vanishing in Proposition \ref{prop:vanishing}, we now turn to the Euler characteristics of the terms in the Buchsbaum--Rim complex. The following lemma packages a computational ingredient into a coefficient-extraction formula adapted to Hirzebruch--Riemann--Roch. For a power series $P(z_1, \ldots, z_n)$, we denote by $[P(z_1, \ldots, z_n)]_{z_1^{a_1}\cdots z_n^{a_n}}$ the coefficient of the monomial $z_1^{a_1}\cdots z_n^{a_n}$ in $P(z_1, \ldots, z_n)$.

\begin{lemma}\label{lem:powerseries}
Let $\lambda_1, \ldots, \lambda_p$ be the Chern roots of $E_{p, \omega_C}^{\vee}$. For formal power series $\Phi(z)$ and $\Psi(z)$ such that $\Phi(0)$ is invertible, we have
$$
\int_{C_p} \Psi(x_p) \prod_{i=1}^p \Phi(\lambda_i) = \left[ \Psi(u) \Phi(0)^{g-1} \Phi(u)^{p-g+1} \left( 1 - u \frac{\Phi'(u)}{\Phi(u)} \right) ^ g \right]_{u^p}.
$$
\end{lemma}

\begin{proof}
Recall that $E_{p, \omega_C}^{\vee} = \Theta_{C_p}$ is the tangent bundle of $C_p$.
By \cite[Lemma VIII.2.5]{ACGH85}, 
$$
c_t(E_{p, \omega_C}^{\vee}) =  (1+\lambda_1t) \cdots (1+ \lambda_p t) = (1+x_pt)^{p-g+1} e^{-\frac{\theta_p t}{1+x_pt}}.
$$
Taking logarithms and comparing the homogeneous terms of degree $n$ in $t$, we find
$$
P_n:=\sum_{i=1}^p \lambda_i^n = (p-g+1)x_p^n - n \theta_p x_p^{n-1}~~\text{ for any integer $n \geq 1$}
$$
Set
$$
q(z):=\log \frac{\Phi(z)}{\Phi(0)}.
$$
Since $q(0)=0$, the above power-sum identity gives
$$
\sum_{i=1}^p q(\lambda_i) = (p-g+1)q(x_p) - \theta_p q'(x_p). 
$$
Thus
$$
\prod_{i=1}^p \Phi(\lambda_i) = \Phi(0)^p e^{\sum_{i=1}^p q(\lambda_i)}
= \Phi(0)^{g-1} \Phi(x_p)^{p-g+1} e^{-\theta_p \frac{\Phi'(x_p)}{\Phi(x_p)}}.
$$
For formal power series $F(u)$ and $R(u)$, the intersection formula (\ref{eq:xtheta}) yields
\begin{align*}
\int_{C_p} F(x_p) e^{\theta_p R(x_p)} &= \sum_{j=0}^{\min\{p,g\}} \frac{1}{j!} \frac{g!}{(g-j)!} [F(u)R(u)^j]_{u^{p-j}}\\
&= \left[ F(u) \sum_{j=0}^g {g \choose j} (u R(u))^j \right]_{u^{p}}\\
&= [F(u)(1+uR(u))^g]_{u^p}.
\end{align*}
Applying this identity with
$$
F(u)=\Psi(u)\Phi(0)^{g-1}\Phi(u)^{p-g+1}~~\text{ and }~~R(u)=-\frac{\Phi'(u)}{\Phi(u)},
$$
we obtain the assertion. 
\end{proof}

\begin{proposition}\label{prop:eulercharacteristic}
Let $A$ be a line bundle on $C$, and $a:=\deg A$. For $k \leq p \leq g-1$, we have
$$
\sum_{i=0}^{g-p-1} (-1)^i {g \choose p+1+i} \chi(S^i E_{p, \omega_C}^{\vee} \otimes S_{p, A}) = \sum_{i=0}^{g-p-1} {g \choose i} \left( (g-p-i) {a-2 \choose p-i} + {a-2 \choose p-i-1} \right).
$$
\end{proposition}

\begin{proof}
By Hirzebruch--Riemann--Roch theorem, 
$$
h_i(a):=\chi(S^i E_{p, \omega_C}^{\vee} \otimes S_{p, A}) = \int_{C_p} \ch(S^i E_{p, \omega_C}^{\vee} \otimes S_{p, A}) \cdot \td(C_p) = \int_{C_p} e^{ax_p} \ch(S^i E_{p, \omega_C}^{\vee} ) \td(C_p).
$$
Let $\lambda_1, \ldots, \lambda_p$ be the Chern roots of $E_{p, \omega_C}^{\vee}$.
We have
$$
\sum_{i=0}^{\infty} \ch(S^i E_{p, \omega_C}^{\vee}) w^i = \prod_{i=1}^p \frac{1}{1-we^{\lambda_i}}~~\text{ and }~~\td(C_p) = \prod_{i=1}^p \frac{\lambda_i}{1-e^{-\lambda_i}}. 
$$
Then 
$$
H_a(w):=\sum_{i=0}^{\infty} h_i(a) w^i = \int_{C_p} e^{ax_p} \left( \sum_{i=0}^{\infty} \ch(S^i E_{p, \omega_C}^{\vee}) w^i \right) \td(C_p) = \int_{C_p} e^{ax_p} \prod_{i=1}^p \Phi_w(\lambda_i),
$$
where
$$
\Phi_w(z):=\frac{z}{1-e^{-z}} \frac{1}{1-we^z}~~\text{ with }~~\Phi_w(0)=\frac{1}{1-w}.
$$
By Lemma \ref{lem:powerseries} for $\Phi(u)=\Phi_w(u)$ and $\Psi(u)=e^{au}$, we obtain
$$
H_a(w)=\left[e^{au} \Phi_w(0)^{g-1} \Phi_w(u)^{p-g+1} \left( 1 - u \frac{\Phi_w'(u)}{\Phi_w(u)}\right)^g \right]_{u^p}.
$$
Logarithmic differentiation gives
$$
1-u\frac{\Phi_w'(u)}{\Phi_w(u)} = 1 - u \left(\frac{1}{u} - \frac{1}{e^u-1} + \frac{we^u}{1-we^u} \right) = \frac{u(1-we^{2u})}{(e^u-1)(1-we^u)}. 
$$
Thus
$$
H_a(w) = \left[e^{(a+p-g+1)u} \left( \frac{u}{e^u-1} \right)^{p+1} (1-w)^{1-g} (1-we^u)^{-p-1}(1-we^{2u})^g \right]_{u^p}.
$$
On the other hand, notice that
$$
\sum_{i=0}^{g-p-1} (-1)^i {g \choose p+1+i} \chi(S^i E_{p, \omega_C}^{\vee} \otimes S_{p, A})=[(1+w)^g H_a(-w)]_{w^{g-p-1}}.
$$
By the formal residue change, we get
\begin{align*}
&\sum_{i=0}^{g-p-1} (-1)^i {g \choose p+1+i} \chi(S^i E_{p, \omega_C}^{\vee} \otimes S_{p, A})\\
&= \left[ e^{(a+p-g+1)u} \left( \frac{u}{e^u-1} \right)^{p+1} (1+w) (1+we^u)^{-p-1}(1+we^{2u})^g\right]_{w^{g-p-1}u^p}\\
&= [y^{a-g+p}(1+w)(1+wy)^{-p-1}(1+wy^2)^g ]_{w^{g-p-1}z^p},
\end{align*}
where $z:=e^u-1$ and $y:=1+z$. 
As
$$
1+wy^2 = (1+wy)\left( 1+ \frac{wzy}{1+wy}\right),
$$
we find
$$
(1+wy)^{-p-1}(1+wy^2)^g = \left( 1+ \frac{wzy}{1+wy} \right)^g (1+wy)^{g-p-1}= \sum_{i=0}^g {g \choose i} w^i z^i y^i (1+wy)^{g-p-1-i}.
$$
Note that 
$$
[(1+w)(1+wy)^{g-p-1-i}]_{w^{g-p-1-i}} = y^{g-p-1-i} + (g-p-1-i)y^{g-p-2-i}.
$$
Hence
\begin{align*}
&\sum_{i=0}^{g-p-1} (-1)^i {g \choose p+1+i} \chi(S^i E_{p, \omega_C}^{\vee} \otimes S_{p, A})\\
&= \sum_{i=0}^{g-p-1} {g \choose i} [ y^{a-g+p} (y^{g-p-1} + (g-p-1-i)y^{g-p-2})]_{z^{p-i}}\\
&= \sum_{i=0}^{g-p-1} {g \choose i} [ (1+z)^{a-g+p} ((1+z)^{g-p-1} + (g-p-1-i)(1+z)^{g-p-2)}]_{z^{p-i}}\\
&= \sum_{i=0}^{g-p-1} {g \choose i} \left( {a-1 \choose p-i} + (g-p-1-i){a-2 \choose p-i} \right).
\end{align*}
Now, using Pascal identity, one can easily check that
$$
{a-1 \choose p-i} + (g-p-1-i){a-2 \choose p-i}
= (g-p-i) {a-2 \choose p-i} + {a-2 \choose p-i-1}.
$$
This finishes the proof. 
\end{proof}

The following theorem readily implies Theorem \ref{thm:main1} since $C_p^1$ has the expected dimension $2p-g-1$ for all $k \leq p \leq g-1$ when $C$ has maximal gonality $k=\lfloor (g+3)/2 \rfloor$ and $\dim W_k^1(C)=0$ if $g$ is even. 

\begin{theorem}\label{thm:bettinumber}
For an integer $m$ with $k \leq m \leq g-1$, assume that $\dim C_m^1 = 2m-g-1$ and $H^1(C, L \otimes \omega_C^{-1})=0$. For $m \leq p \leq g-1$, we have
$$
\kappa_{r-p, 1}(C, L)= \sum_{i=g-p}^p {g \choose i} \left( (i-g+p) {d- 2g \choose p-i} - {d - 2g \choose p-i-1} \right)
$$
and
$$
\kappa_{r-p-1, 2}(C, L) = \sum_{i=0}^{g-p-1} {g \choose i} \left( (g-p-i) {d-2g \choose p-i} + {d-2g \choose p-i-1} \right).
$$
\end{theorem}

\begin{proof}
By \cite[Theorem 1]{FHL84}, $C_p^1$ has the expected dimension $2p-g-1$ for $m \leq p \leq g-1$.
Consider the evaluation map 
$$
\operatorname{ev}_{p, \omega_C} \colon H^0(C, \omega_C) \otimes \mathscr{O}_{C_p} \longrightarrow E_{p, \omega_C}, 
$$
and let $\mathscr{F}_{p, \omega_C}$ be the kernel of $\operatorname{ev}_{p, \omega_C}$. Recall that
$$
\kappa_{r-p-1, 2}(C, L) = \kappa_{p,0}(C, \omega_C; L) = h^0(C_p, \mathscr{F}_{p, \omega_C} \otimes N_{p,L}). 
$$
The degeneracy locus of $\ev_{p, \omega_C}$ is $C_p^1$, which is assumed to have the expected dimension $2p-g-1$, so it has the expected codimension $g-p+1$. By \cite[Theorem B.2.2]{Lazarsfeld04} (see also \cite[Theorem A2.10]{Eisenbud95}), the kernel $\mathscr{F}_{p, \omega_C}$ of $\ev_{p, \omega_C}$ is resolved by the Buchsbaum--Rim (or Eagon--Northcott) complex $\operatorname{BR}_{p,\bullet}$:
\begin{align*}
& 0 \longrightarrow \wedge^g H^0(C, \omega_C) \otimes S^{g-p-1} E_{p, \omega_C}^{\vee} \otimes N_{p, \omega_C}^{-1} \longrightarrow \wedge^{g-1} H^0(C, \omega_C) \otimes S^{g-p-2} E_{p, \omega_C}^{\vee} \otimes N_{p, \omega_C}^{-1} \longrightarrow  \\
& \quad \quad \quad \quad \quad \cdots \longrightarrow \wedge^{p+2} H^0(C, \omega_C) \otimes E_{p, \omega_C}^{\vee} \otimes N_{p, \omega_C}^{-1} \longrightarrow \wedge^{p+1} H^0(C, \omega_C) \otimes N_{p, \omega_C}^{-1} \longrightarrow \mathscr{F}_{p, \omega_C} \longrightarrow 0. 
\end{align*}
Notice that $N_{p, \omega_C}^{-1} \otimes N_{p,L} = S_{p, L \otimes \omega_C^{-1}} = S_{p,A}$, where $A:=L \otimes \omega_C^{-1}$. 
Since $H^1(C, A)=0$, Proposition \ref{prop:vanishing} shows that
$$
H^i(C_p, S^j E_{p, \omega_C}^{\vee} \otimes S_{p,A})=0~~\text{ for $i>0$ and $0 \leq j \leq i-1$}. 
$$
Chasing through the Buchsbaum--Rim complex (see \cite[Proposition B.1.2]{Lazarsfeld04}), we obtain
$$
H^i(C_p, \mathscr{F}_{p, \omega_C} \otimes N_{p,L})=0~~\text{ for $i>0$}.
$$
Thus
$$
\kappa_{r-p-1, 2}(C, L) = h^0(C_p, \mathscr{F}_{p, \omega_C} \otimes N_{p,L}) = \chi( \mathscr{F}_{p, \omega_C} \otimes N_{p,L}).
$$
The Buchsbaum--Rim complex and Proposition \ref{prop:eulercharacteristic} yield
\begin{align*}
\chi( \mathscr{F}_{p, \omega_C} \otimes N_{p,L}) &= \sum_{i=0}^{g-p-1} (-1)^i {g \choose p+1+i} \chi(S^i E_{p, \omega_C}^{\vee} \otimes S_{p, A})\\
&=\sum_{i=0}^{g-p-1} {g \choose i} \left( (g-p-i) {d-2g \choose p-i} + {d-2g \choose p-i-1} \right).
\end{align*}
Now, we recall from (\ref{eq:k_1-k_2}) that
$$
\kappa_{r-p,1}(C,L) = \kappa_{r-p-1,2}(C, L) + (d-g-p) { d - g \choose p} \left( \frac{d-g+1}{d-g-p+1} - \frac{d}{d-g} \right).
$$
It is easy to see that
$$
(d-g-p) { d - g \choose p} \left( \frac{d-g+1}{d-g-p+1} - \frac{d}{d-g} \right)
= g {d-g-1 \choose p-1} - (g-p) {d-g \choose p} - {d-g \choose p-1}.
$$
Consider the polynomial
$$
P(t):=Q_1(t)-Q_2(t)
$$
where
$$
Q_1(t):=\sum_{i=g-p}^g {g \choose i} \big( (i-g+p)t^i - t^{i+1} \big)~~\text{ and }~~Q_2(t):=\sum_{i=0}^{g-p-1}{g \choose i} \big( (g-p-i)t^i + t^{i+1} \big).
$$
Notice that
$$
\kappa_{r-p-1,2}(C,L)=[(1+t)^{d-2g}Q_2(t)]_{t^p}.
$$
Since
$$
P(t)=\sum_{i=0}^g {g \choose i} \big( (i-g+p)t^i - t^{i+1} \big) = gt(1+t)^{g-1} - (g-p)(1+t)^g - t(1+t)^g,
$$
it follows that
$$
[(1+t)^{d-2g}P(t)]_{t^p} = g {d-g-1 \choose p-1} - (g-p) {d-g \choose p} - {d-g \choose p-1}.
$$
Thus
$$
\kappa_{r-p,1}(C,L) = [(1+t)^{d-2g} (Q_2(t)+P(t))]_{t^p} = [(1+t)^{d-2g} Q_1(t) ]_{t^p}. 
$$
This proves
\[
\kappa_{r-p,1}(C,L) = \sum_{i=g-p}^p {g \choose i} \left( (i-g+p) {d- 2g \choose p-i} - {d - 2g \choose p-i-1} \right). \qedhere
\]
\end{proof}

\begin{remark}\label{rem:leadcoff}
Assume that $d =\deg L \gg 0$, and let $\mathscr{G}_{p, \omega_C}$ be the cokernel of the evaluation map $\operatorname{ev}_{p, \omega_C} \colon H^0(C, \omega_C) \otimes \mathscr{O}_{C_p} \rightarrow E_{p, \omega_C}$. Suppose that $\dim C_p^1=2p-g-1$ for an integer $p$ with $k \leq p \leq g-1$. We briefly explain an alternative approach to computing the leading coefficient of $\kappa_{r-p,1}(C,L)$, regarded as a polynomial in $d$. 
As $d =\deg L \gg0$, we find
$$
\kappa_{r-p,1}(C,L) = h^0(C_p, \mathscr{G}_{p, \omega_C} \otimes N_{p,L}) = \chi(\mathscr{G}_{p, \omega_C} \otimes N_{p,L}). 
$$
Note that $\mathscr{G}_{p, \omega_C}$ is supported on $C_p^1$ and has generic rank one along its support. By \cite[p.326]{ACGH85}, the fundamental class of $C_p^1$ is 
\begin{equation}\label{eq:fundclass}
[C_p^1] = \frac{\theta_p^{g-p}}{(g-p+1)!}(\theta_p - (g-p+1)x_p). 
\end{equation}
Applying the asymptotic Riemann--Roch theorem \cite[Example 1.2.19]{Lazarsfeld04} and (\ref{eq:xtheta}), one can show that the leading coefficient of $\kappa_{r-p,1}(C,L)$ is
$$
\frac{[C_p^1] \cdot x_p^{2p-g-1}}{(2p-g-1)!} = \frac{1}{p(2p-g-2)!} {g \choose p-1}.
$$
This approach is also useful without assuming that $C_p^1$ has the expected dimension in several special cases. 
First, suppose that $W_p^1(C)$ consists of finitely many reduced points. Set $n:=\# W_p^1(C)$. Then $C_p^1$ is a disjoint union of $n$ copies of $\nP^1$. On each component, 
$\mathscr{G}_{p, \omega_C}|_{\nP^1} = \omega_{\nP^1}$ and $N_{p,L}|_{\nP^1} = \sO_{\nP^1}(r-p+1)$. Thus 
$$
\kappa_{r-p,1}(C,L) = n(r-p). 
$$
This argument was sketched in \cite[Proposition 2.3]{EL15}. If $C$ is a general curve of gonality $k$ with $2 \leq k \leq \lfloor (g+1)/2 \rfloor$, then $W_k^1(C)$ is a single reduced point by \cite{AC81} so that
$$
\kappa_{r-k, 1}(C,L) = r-k. 
$$
This recovers Kemeny's result \cite[Theorem 1.1]{Kemeny20} in the case $d \gg 0$. If $C$ is a general curve of even genus $g=2k-2$, then 
$$
\kappa_{r-k,1}(C,L) = \frac{1}{k}{2k-2 \choose k-1} (r-k). 
$$
Here the leading coefficient is the number of points in $W_k^1(C)$ (see \cite[p.211]{ACGH85}). 
Next, suppose that $C$ is a general curve of odd genus $g=2k-3$. In this case, $B:=W_k^1(C)$ is a smooth curve of genus
$$
\widetilde{g}:=1 + \frac{k-2}{k} {2k-2 \choose k-1},
$$
(see \cite[Theorem 4]{EH87}), and $S:=C_k^1$ is a smooth surface such that the restricted Abel--Jacobi map $u_k \colon S \to B$ is a $\nP^1$-fibration. We know that $\mathscr{G}_{k, \omega_C} = \omega_{C_k^1/W_k^1(C)}$. By Riemann--Roch theorem,
$$
\kappa_{r-k,1}(C,L) 
= \chi(\omega_{S/B} \otimes N_{k,L}|_S)
= \frac{1}{2}\big([K_{S/B}] + [N_{k,L}|_{S}]\big)\big([K_{S/B}] + [N_{k,L}|_{S}] - [K_{S}]\big) + 1 - \widetilde{g}. 
$$
Let $\overline{x}_k:=x_k|_{S}$ and $\overline{\theta}_k:=\theta_k|_{S}$. By \cite[Lemma 5.8]{AN10}, $[N_{k,L}|_{S}] = (d-3k+4)\overline{x}_k + \overline{\theta}_k$. For a general point $z \in C$, the divisor $X_{k,z}|_{S}$ maps isomorphically onto $B$ under $u_k$. 
If $F$ is the divisor class of a fiber of  $u_k \colon S \to B$, then $\overline{x}_k \cdot F=1$ and $\overline{\theta}_k \cdot F=0$. Furthermore, by adjunction formula, 
$$
[K_{S}] \cdot \overline{x}_k = 2\widetilde{g}-2 - \overline{x}_k^2~\text{ and }~[K_{S}] \cdot \overline{\theta}_k = -2\overline x_k \overline{\theta}_k. 
$$
By a direct calculation using (\ref{eq:xtheta}) and (\ref{eq:fundclass}), we obtain
$$
\kappa_{r-k,1}(C,L) = \frac{1}{k}{2k-3 \choose k-1}(d-3k+3)(d-4k+8).
$$
Here the leading coefficient is $(\tilde{g}-1)/(g-1)$.
\end{remark}

\begin{remark}
By the method introduced in this section, we can also compute $\kappa_{p,0}(C, B; L)$ and $\kappa_{p-1,1}(C,B;L)$ for any line bundle $B$ on $C$ provided that $H^0(C, B \otimes L^{-1})=H^1(C, L \otimes B^{-1})=0$ and the degeneracy locus of the evaluation map $\ev_{p, B} \colon H^0(C, B) \otimes \sO_{C_p} \to E_{p,B}$ has the expected dimension. For instance, if we further assume that $H^1(C, B)=0$ and $b:=\deg B$, then 
$$
\kappa_{p,0}(C, B; L) = \sum_{i=0}^{b-g+1-p} (b-g+1-p-i)  {g \choose i} {d-2g \choose p-i}
$$
and 
$$
\kappa_{p-1,1}(C,B;L) = \sum_{i=b-g+1-p}^p (i-b+g-1+p) {g \choose i} {d - 2g \choose p-i}.
$$
We leave the details to the interested reader. 
\end{remark}

Next, we turn to the proofs of Corollary \ref{cor:inequality} and Theorem \ref{thm:main2}. First, we present the result of the computation of the graded Betti numbers in the case where $C$ is a hyperelliptic curve. This is essentially due to Nagel--Pittleloud \cite[Proposition 2.3]{NP94}.

\begin{theorem}[Nagel--Pittleloud]\label{thm:bettihyperelliptic}
Assume that $C$ is a hyperelliptic curve and $d:=\deg L \geq 2g+1$. For $2 \leq p \leq g-1$, we have
$$
\kappa_{r-p,1}(C, L)=(d-g-p){d-g-1 \choose p-2}
~\text{ and }~
\kappa_{r-p-1,2}(C,L) = (g-p){d-g-1 \choose p}.
$$
\end{theorem}

\begin{proof}
Consider the embedding $C \subseteq \nP H^0(C, L) = \nP^r$. As $C \subseteq \nP^r$ is arithmetically Cohen--Macaulay, its graded Betti numbers are the same as those of the homogeneous coordinate ring of its general hyperplane section $\Gamma \subseteq \nP^{r-1}$. On the other hand, the lines spanned by the fibers of the hyperelliptic fibration $f \colon C \to \nP^1$ sweep out a rational normal surface scroll 
$$
S:=\bigcup_{t \in \nP^1} \langle f^{-1}(t) \rangle \subseteq \nP^r. 
$$
Then $\Gamma$ is contained in a general hyperplane section of $S \subseteq \nP^r$, which is a rational normal curve in $\nP^{r-1}$. The graded Betti numbers of a finite subscheme of a rational normal curve are computed in \cite[Proposition 2.3]{NP94}. 
\end{proof}

Corollary \ref{cor:inequality} and Theorem \ref{thm:main2} follow at once from Theorem \ref{thm:bettihyperelliptic} and the previous discussion.

\begin{proof}[Proof of Corollary \ref{cor:inequality}]
Recall that we assume $d \gg 0$ and $\kappa_{r-p,1}(C,L)$ is a polynomial in $d$ whose degree is $\dim C_p^1$ and the leading coefficient is positive.
When $C_p^1$ has the expected dimension, the degree of $\kappa_{r-p,1}(C,L)$ is minimal, so the left inequality follows from Theorem \ref{thm:bettinumber}. On the other hand, if $C_p^1$ has the largest possible dimension, then $C$ is hyperelliptic. In this case, the degree of $\kappa_{r-p,1}(C,L)$ is maximal, and hence, the right inequality follows from Theorem \ref{thm:bettihyperelliptic}. The characterization of the equality cases for $\kappa_{r-p,1}(C,L)$ is also easily deduced. Finally, for fixed $d$ and $g$, the difference $\kappa_{r-p,1}(C,L) - \kappa_{r-p-1,2}(C,L)$ is constant. This implies the corresponding statements for $\kappa_{r-p-1,2}(C,L)$. 
\end{proof}

The following theorem, together with Theorem \ref{thm:bettihyperelliptic}, proves Theorem \ref{thm:main2}.

\begin{theorem}
Assume that $C$ is not hyperelliptic and $H^1(C, L)=0$. Then 
$$
\kappa_{r-g, 2}(C, L) = h^0(C_{g-1}, S_{g-1, A})= { h^0(C, L \otimes \omega_C^{-1}) + g - 2 \choose g-1}.
$$
\end{theorem}

\begin{proof}
Recall that $\kappa_{r-g, 2}(C, L) = h^0(C_{g-1}, \mathscr{F}_{g-1, \omega_C} \otimes N_{g-1,L})$. If $g=3$, then $\ev_{2, \omega_C}$ is surjective and $\mathscr{F}_{2, \omega_C}=N_{2, \omega_C}^{-1}$ directly. For $g \geq 4$, by Martens theorem (see e.g., \cite[Theorem IV.5.1]{ACGH85}), $\dim C_{g-1}^1$ has the expected dimension $g-3$. Then the Buchsbaum--Rim complex $\operatorname{BR}_{g-1,\bullet}$ yields that $\mathscr{F}_{g-1, \omega_C} = N_{g-1, \omega_C}^{-1}$. Thus $\kappa_{r-g, 2}(C, L) = h^0(C_{g-1}, S_{g-1, L \otimes \omega_C^{-1}})$, which finishes the proof.
\end{proof}

As an application of Theorem \ref{thm:main2}, we compute the first nonzero graded Betti number of weight 2 for a rational surface with effective anticanonical divisor. The following provides a refinement of \cite[Proposition 3.4]{EGHP05} under a slightly stronger assumption.

\begin{corollary}\label{cor:rationalsurface}
Let $S$ be a smooth projective complex rational surface with $H^0(S, -K_S) \neq 0$, and $H$ be a very ample divisor on $S$ such that $K_S + H$ is base point free and big and $H^i(S, mH)=0$ for $i>0$ and $m>0$. If $-K_S.H \geq 3$, then $\sO_S(H)$ satisfies property $N_{-K_S.H-3}$, i.e., $K_{p,2}(S, \sO_S(H)) = 0$ for $0 \leq p \leq -K_S. H - 3$, and 
$$
\kappa_{-K_S.H-2, 2}(S, \sO_S(H)) = {h^0(S, -K_S) + g-2 \choose g-1},
$$
where $g:=1/2(K_S+H).H+1$ is the genus of a smooth curve in $\lvert H \rvert$. In particular, we have the following:
\begin{enumerate}
    \item $\sO_{\nP^2}(d)$ for $d \geq 4$ satisfies property $N_{3d-3}$, and 
    $$
    \kappa_{3d-2,2}(\nP^2, \sO_{\nP^2}(d)) = { \frac{d^2-3d+18}{2} \choose 9}.
    $$
    \item $\sO_{\nP^1}(a) \boxtimes \sO_{\nP^1}(b)$ for $a,b \geq 3$ satisfies property $N_{2a+2b-3}$, and 
    $$
    \kappa_{2a+2b-2,2}(\nP^1 \times \nP^1, \sO_{\nP^1}(a) \boxtimes \sO_{\nP^1}(b))= {(a-1)(b-1)+7 \choose 8}.
    $$
\end{enumerate}
\end{corollary}

\begin{proof}
Let $C \in \lvert H \rvert$ be a general smooth curve, and $L:=\sO_C(H|_C)$. 
By \cite[Theorem 3.b.7]{Green84} (see also \cite[Theorem 2.21]{AN10}), $K_{p,2}(S, \sO_S(H)) = K_{p,2}(C,L)$ for $p \geq 0$. Note that
$$
\deg L = H^2 = (K_S+H).H - K_S.H = 2g+1 + (-K_S.H-3). 
$$
By Green's $(2g+1+p)$-theorem, $K_{p,2}(C,L)=0$ for $0 \leq p \leq -K_S.H-3$. Observe that $-K_S|_C = (-K_S-H)_C + H|_C = -K_C + H|_C$. 
We have a short exact sequence
$$
0 \longrightarrow \sO_S(-K_S - H) \longrightarrow \sO_S(-K_S) \longrightarrow L \otimes \omega_C^{-1} \longrightarrow 0. 
$$
By Kawamata--Viehweg vanishing theorem, $H^0(S, -K_S - H) = H^1(S, -K_S-H)=0$. Thus $H^0(C, L \otimes \omega_C^{-1})=H^0(S, -K_S)$. We can easily check that $C$ is not hyperelliptic using \cite[Theorem 11.4.2]{BS95}. Then the assertion follows from Theorem \ref{thm:main2}.
\end{proof}

The result that $\sO_{\nP^2}(d)$ for $d \geq 3$ satisfies property $N_{3d-3}$ but fails to satisfy property $N_{3d-2}$ was established in \cite[Theorem 2.1 and Proposition 3.2]{OP01}. The exact value of $\kappa_{3d-2,2}(\nP^2, \sO_{\nP^2}(d))$ was conjectured in \cite[Conjecture 1.6]{CCDL19} and verified in \cite[Theorem 4]{Lemmens18}. Our proof of property $N_p$ is essentially the same as that in \cite{OP01}, but our computation of the first nonzero graded Betti number of weight 2 is entirely different from the existing approaches.

\begin{remark}
According to adjunction theory (see e.g., \cite{BS95}), if $S$ is a smooth projective complex surface and $H$ is a very ample divisor on $S$, then $K_S + H$ is base point free and big in most cases.
\end{remark}

Finally, we compute $\kappa_{r-g+1,2}(C, L)$, the second graded Betti number that may be nonzero, under the condition that $C$ is neither hyperelliptic, trigonal, bielliptic, nor a smooth plane quintic. In particular, $g \geq 5$. By Mumford theorem (see e.g., \cite[Theorem IV.5.2]{ACGH85}), this condition guarantees that $C_{g-2}^1$ has the expected dimension. 
When $H^0(C, \omega_C^2 \otimes L^{-1})=H^1(C, L \otimes \omega_C^{-1})^{\vee}=0$, the exact value of $\kappa_{r-g+1, 2}(C, L)$ is given in Theorem \ref{thm:bettinumber}. We only need to consider the case that $H^0(C, \omega_C^2 \otimes L^{-1}) \neq 0$. But we assume a stronger condition as follows.

\begin{theorem}\label{thm:kappa_{r-g+1,2}}
Assume the following:
\begin{enumerate}
    \item $C$ is neither hyperelliptic, trigonal, bielliptic, nor a smooth plane quintic.
    \item $L$ is nonspecial, i.e., $H^1(C, L)=0$. 
    \item $\omega_C^2 \otimes L^{-1}$ is base point free, and either $H^1(C, \omega_C^2 \otimes L^{-1})=0$ or $h^0(C, \omega_C^2 \otimes L^{-1}) \geq 3$ and the multiplication map $H^0(C, \omega_C) \otimes H^0(C, \omega_C^2 \otimes L^{-1}) \rightarrow H^0(C, \omega_C^3 \otimes L^{-1})$ is surjective.
\end{enumerate}
Then 
$$
\kappa_{r-g+1, 2}(C, L) = g { h^0(C, L \otimes \omega_C^{-1})+g-3 \choose g-2}.
$$
\end{theorem}

\begin{proof}
By Mumford theorem (see e.g., \cite[Theorem IV.5.2]{ACGH85}), the condition $(1)$ implies that $\dim C_{g-2}^1 = g-5 = 2(g-2)-g-1$. Then the Buchsbaum--Rim complex $\operatorname{BR}_{g-2,\bullet}$ yields a short exact sequence
$$
0 \longrightarrow E_{g-2, \omega_C}^{\vee} \otimes N_{g-2, \omega_C}^{-1} \longrightarrow H^0(C, \omega_C)^{\vee} \otimes N_{g-2, \omega_C}^{-1} \longrightarrow \mathscr{F}_{g-2, \omega_C} \longrightarrow 0. 
$$
By the condition $(2)$, we have
$$
\kappa_{r-g+1, 2}(C, L) = 
h^0(C_{g-2}, \mathscr{F}_{g-2, \omega_C} \otimes N_{g-2, L}).
$$
We show that the condition $(3)$ implies that the map
$$
H^1(C_{g-2}, E_{g-2, \omega_C}^{\vee} \otimes S_{g-2, L \otimes \omega_C^{-1}}) \longrightarrow H^0(C, \omega_C)^{\vee} \otimes H^1(C_{g-2}, S_{g-2, L \otimes \omega_C^{-1}})
$$
is injective. It is equivalent to that the dual map
$$
H^0(C, \omega_C) \otimes H^{g-3}(C_{g-2}, N_{g-2, \omega_C^2 \otimes L^{-1}}) \longrightarrow H^{g-3}(C_{g-2}, E_{g-2, \omega_C} \otimes N_{g-2, \omega_C^2 \otimes L^{-1}})
$$
is surjective. Considering the Leray spectral sequence for the second projection $\pr_2 \colon C_{g-3} \times C \to C$ and applying \cite[Lemma 2.4]{NP24+}, we find
\begin{equation}\label{eq:H^iC_{g-2}}
\begin{aligned}
&H^{i}(C_{g-2}, E_{g-2, \omega_C} \otimes N_{g-2, \omega_C^2 \otimes L^{-1}})=H^{i}(C_{g-3} \times C, (N_{g-3, \omega_C^2 \otimes L^{-1}} \boxtimes \omega_C^3 \otimes L^{-1})(-D_{g-3, 1}))\\
&= 
\begin{array}{l}
(H^1(C, \wedge^{g-2-i}M_{\omega_C^2 \otimes L^{-1}} \otimes \omega_C^3 \otimes L^{-1}) \otimes S^{i-1} H^1(C, \omega_C^2 \otimes L^{-1}))\\[5pt]
\oplus (H^0(C, \wedge^{g-3-i}M_{\omega_C^2 \otimes L^{-1}} \otimes \omega_C^3 \otimes L^{-1}) \otimes S^{i} H^1(C, \omega_C^2 \otimes L^{-1}))
\end{array}~~ \text{ for $i \geq 0$}.
\end{aligned}
\end{equation}
By $(\ref{eq:H^iC_{g-2}})$ for $i=g-3$, we have
$$
 H^{g-3}(C_{g-2}, E_{g-2, \omega_C} \otimes N_{g-2, \omega_C^2 \otimes L^{-1}})=\begin{array}{l}
(H^1(C, M_{\omega_C^2 \otimes L^{-1}} \otimes \omega_C^3 \otimes L^{-1}) \otimes S^{g-4} H^1(C, \omega_C^2 \otimes L^{-1}))\\[5pt]
\oplus (H^0(C, \omega_C^3 \otimes L^{-1}) \otimes S^{g-3} H^1(C, \omega_C^2 \otimes L^{-1})).
\end{array}
$$
If $H^1(C, \omega_C^2 \otimes L^{-1})=0$, then $H^{g-3}(C_{g-2}, E_{g-2, \omega_C} \otimes N_{g-2, \omega_C^2 \otimes L^{-1}})=0$ so that the above dual map is clearly surjective. We assume that $H^1(C, \omega_C^2 \otimes L^{-1}) \neq 0$. Then $r(\omega_C^2 \otimes L^{-1}) \geq 2$. Note that
\begin{align*}
&H^{g-3}(C_{g-2},N_{g-2, \omega_C^2 \otimes L^{-1}})
= H^0(C, \omega_C^2 \otimes L^{-1}) \otimes S^{g-3} H^1(C, \omega_C^2 \otimes L^{-1})\\
&H^1(C, M_{\omega_C^2 \otimes L^{-1}} \otimes \omega_C^3 \otimes L^{-1})
= H^0(C, \wedge^{r(\omega_C^2 \otimes L^{-1})-1} M_{\omega_C^2 \otimes L^{-1}})^{\vee}=0. 
\end{align*}
Thus the above dual map is obtained by tensoring the multiplication map
$$
H^0(C, \omega_C) \otimes H^0(C, \omega_C^2 \otimes L^{-1}) \longrightarrow H^0(C, \omega_C^3 \otimes L^{-1})
$$
with $S^{g-3} H^1(C, \omega_C^2 \otimes L^{-1})$, so it is surjective by the condition $(2)$. Now, we obtain
$$
h^0(C_{g-2}, \mathscr{F}_{g-2, \omega_C} \otimes N_{g-2, L})
= g h^0(C_{g-2}, S_{g-2, L \otimes \omega_C^{-1}}) - h^0(C_{g-2}, E_{g-2, \omega_C}^{\vee} \otimes S_{g-2, L \otimes \omega_C^{-1}}).
$$
By (\ref{eq:H^iC_{g-2}}) for $i=g-2$, we have
\begin{align*}
& H^0(C_{g-2}, E_{g-2, \omega_C}^{\vee} \otimes S_{g-2, L \otimes \omega_C^{-1}})
 = H^{g-2}(C_{g-2}, E_{g-2, \omega_C} \otimes N_{g-2, \omega_C^2 \otimes L^{-1}})^{\vee}\\
& = H^1(C, \omega_C^3 \otimes L^{-1})^{\vee} \otimes S^{g-3} H^1(C, \omega_C^2 \otimes L^{-1})^{\vee}.
\end{align*}
As $\deg (\omega_C^2 \otimes L^{-1})>0$, we get $ H^1(C, \omega_C^3 \otimes L^{-1})=0$. Thus the assertion follows. 
\end{proof}

\begin{remark}
Assume that $\omega_C^2 \otimes L^{-1}$ is base point free and $H^1(C, \omega_C^2 \otimes L^{-1}) \neq 0$. By \cite[Theorem 1.11]{Ciliberto83}, the multiplication map $H^0(C, \omega_C) \otimes H^0(C, \omega_C^2 \otimes L^{-1}) \rightarrow H^0(C, \omega_C^3 \otimes L^{-1})$ is surjective when the image of the morphism induced by $\lvert \omega_C^2 \otimes L^{-1} \rvert$ is not a rational curve. 
\end{remark}

As a quick application of Theorem \ref{thm:kappa_{r-g+1,2}}, we can recover Agostini's result on the secant conjecture \cite[Main Theorem]{Agostini24}.

\begin{corollary}[Agostini]\label{cor:Agostini}
Assume that $C$ is neither hyperelliptic, trigonal, bielliptic, nor a smooth plane quintic. If $\deg L = 2g+p-1$ and $H^1(C, L)=0$, then $L$ is $(p+1)$-very ample if and only if $K_{p,2}(C,L)=0$. 
\end{corollary}

\begin{proof}
If $L$ is not $(p+1)$-very ample, then $K_{p,2}(C,L) \neq 0$ by \cite[Theorem 4.36]{AN10}. Assume that $L$ is $(p+1)$-very ample. By \cite[Lemma 2.2]{Agostini24}, $\omega_C^2 \otimes L^{-1}$ is base point free and $H^1(C, \omega_C^2 \otimes L^{-1})=0$. Then Theorem \ref{thm:kappa_{r-g+1,2}} shows that $K_{p,2}(C,L)=0$ since $r-g+1=(2g+p-1-g)-g+1=p$. 
\end{proof}

\begin{remark}\label{rem:Agostini}
As the Buchsbaum--Rim complex was already used in Agostini's work \cite[Lemma 4.3]{Agostini24}, employing it to study syzygies of algebraic curves is not a new idea in itself. Our contribution is that we use the Buchsbaum--Rim complex to compute the graded Betti numbers exactly when the evaluation map $\ev_{p, \omega_C}$ is not surjective. Moreover, our approach does not encounter the exceptional case that was carefully treated in \cite[Section 5]{Agostini24}. 
\end{remark}

Using Theorem \ref{thm:kappa_{r-g+1,2}}, one can extend Corollary \ref{cor:rationalsurface} to the second potentially nonzero graded Betti numbers of weight 2 to some extent. Here we only focus on the cases of $\nP^2$ and $\nP^1 \times \nP^1$. 

\begin{corollary}
We have
$$
\kappa_{3d-1,2}(\nP^2, \sO_{\nP^2}(d)) = \frac{(d-1)(d-2)}{2} { \frac{d^2-3d+16}{2} \choose 9} ~~ \text{ for $d \geq 7$}
$$
and
$$
\kappa_{2a+2b-1,2}(\nP^1 \times \nP^1, \sO_{\nP^1}(a) \boxtimes \sO_{\nP^1}(b))= (a-1)(b-1) {(a-1)(b-1)+6 \choose 8} ~~\text{ for $a, b \geq 5$}.
$$
\end{corollary}

\begin{proof}
Set $H:=\sO_{\nP^2}(d)$ for $d \geq 7$ or $\sO_{\nP^1}(a) \boxtimes \sO_{\nP^1}(b)$ for $a,b \geq 5$, and let $C \in \lvert H \rvert$ be a general smooth curve and $L:=H|_C$. Then the conditions of Theorem \ref{thm:kappa_{r-g+1,2}} hold. 
Arguing as in the proof of Corollary \ref{cor:rationalsurface}, one can easily show the assertions by applying Theorem \ref{thm:kappa_{r-g+1,2}}.
\end{proof}

\section{Boij--S\"{o}derberg coefficients of algebraic curves}\label{sec:BS}
\noindent This section is devoted to the proof of Theorem \ref{thm:main3}. We always assume that $d=\deg L$ is large enough so that the gonality conjecture holds, i.e., $K_{p,1}(C,L)=0$ for $p \geq r-k+1$ (see \cite[Theorem 1.1]{NP24+}). Let $c_{k-1}, c_k, \ldots, c_g$ be the Boij--S\"{o}derberg coefficients of the section ring $R(C,L)$ as in Subsection \ref{subsec:BS}. Recall from (\ref{eq:BSforcurves}) that
$$
\kappa_{r-p,1}(C,L) = \sum_{i=k-1}^{p-1} c_i \frac{(p-i)(r-p)}{r-i+1} {r+1 \choose p}~~\text{ for $k \leq p \leq g+1$}.
$$
Conversely, we can express the Boij--S\"{o}derberg coefficients in terms of the graded Betti numbers $\kappa_{r-p,1}(C,L)$. Put
$$
\overline{\kappa}_{r-p,1}:= \frac{\kappa_{r-p,1}(C,L)}{(r-p){r+1 \choose p}}~~\text{ for $0 \leq p \leq r-1$}~~\text{ and }~~\overline{c}_i:=\frac{c_i}{r-i+1}~~\text{ for $k-1 \leq i \leq g$}.
$$

\begin{proposition}\label{prop:compBScoeff}
For each $k-1 \leq i \leq g$, we have
$$
c_i = (r-i+1)(\overline{\kappa}_{r-i-1,1}-2\overline{\kappa}_{r-i,1} + \overline{\kappa}_{r-i+1,1}).
$$
In particular, if $d=\deg L \gg 0$, then every $c_i$ is a rational function in $d$. 
\end{proposition}

\begin{proof}
For $k \leq p \leq g+1$, we have
$$
\overline{\kappa}_{r-p,1}= \sum_{i=k-1}^{p-1} (p-i) \overline{c}_i.
$$
Taking two successive differences yields
$$
\overline{\kappa}_{r-p-1,1}-\overline{\kappa}_{r-p,1} = \sum_{i=k-1}^{p} \overline{c}_i~~\text{ for $k-1 \leq p \leq g$}.
$$
Thus we obtain 
$$
\overline{c}_i = \overline{\kappa}_{r-i-1,1}-2\overline{\kappa}_{r-i,1} + \overline{\kappa}_{r-i+1,1}~~\text{ for every $k+1 \leq i \leq g$}.
$$
This also holds for $i=k-1, k$ since $\overline{\kappa}_{r-k+1,1}=\overline{\kappa}_{r-k+2,1}=0$. 
Finally, the last statement follows from \cite[Theorem C]{EL15} saying that $\kappa_{r-p,1}(C,L)$ for $k \leq p \leq g+1$ is a polynomial in $d$ when $d \gg 0$. 
\end{proof}

\begin{corollary}\label{cor:compBScoeff}
The following hold:
\begin{enumerate}
    \item Assume that $C$ is hyperelliptic. If $d \geq 2g+1$, then 
    $$
    c_i=\begin{cases}
    \displaystyle \frac{2(r-i+1)}{r(r+1)} & \text{for $1 \leq i \leq g-1$} \\[10pt]
    \displaystyle \frac{(r-g+1)(r-g+2)}{r(r+1)} & \text{for $i=g$}.
    \end{cases}
    $$
    \item Assume that $\dim C_p^1 = 2p-g-1$ with $k \leq p \leq g-1$ and $H^1(C, L \otimes \omega_C^{-1})=0$. Then 
    $$
    c_i=\frac{{g \choose i}{r-g+2 \choose 2i-g}}{{r+1 \choose i}}~~\text{ for $p+1 \leq i \leq g$}.
    $$
    If $C$ has maximal gonality $\lfloor (g+3)/2 \rfloor$ and $\dim W_k^1(C)=0$ when $g$ is even, then this holds for $k-1 \leq i \leq g$.
\end{enumerate}
\end{corollary}

\begin{proof}
By \cite[Theorem 1.1]{NP24+}, the gonality conjecture holds in the situation under consideration. In view of Proposition \ref{prop:compBScoeff}, it is sufficient to know the exact values of $\overline{\kappa}_{r-p,1}$ for $k-2 \leq p \leq g+1$. For convenience, we carry out the computation under the assumption that $d \geq 2g+2$ (i.e., $r = d-g \geq g+2$). The case $d=2g+1$ can be treated easily. 
Note that $\overline{\kappa}_{r-k+2,1}=\overline{\kappa}_{r-k+1,1}=0$.  By Green's $(2g+1+p)$-theorem and (\ref{eq:k_1-k_2}), we have
$$
\overline{\kappa}_{r-g-1,1} = \frac{r+g^2}{r(r+1)}
~\text{ and }~
\overline{\kappa}_{r-g,1} = \frac{g(g-1)}{r(r+1)}.
$$
First, suppose that $C$ is hyperelliptic so that $k=2$. By Theorem \ref{thm:bettihyperelliptic}, 
$$
\overline{\kappa}_{r-p,1} = \frac{p(p-1)}{r(r+1)}~~\text{ for $0 \leq p \leq g$.}
$$
Then we obtain the assertion $(1)$ from Proposition \ref{prop:compBScoeff}. Next, suppose that $\dim C_p^1 = 2p-g-1$ and $H^1(C, L \otimes \omega_C^{-1})=0$. In this case, $C_m^1$ has the expected dimension for $p \leq m \leq g+1$ by \cite[Theorem 1]{FHL84}.  Recall from  the proof of Theorem \ref{thm:bettinumber} that
$\kappa_{r-m,1}(C,L) = [(1+t)^{r-g}Q_{1,m}(t)]_{t^m}$, where
$$
Q_{1,m}(t):=\sum_{j=g-m}^g {g \choose j} ((j-g+m)t^j - t^{j+1}). 
$$
By Proposition \ref{prop:compBScoeff}, we find
$$
c_i = \frac{1}{{r + 1 \choose i}} \left[ (1+t)^{r-g} \left ( \frac{i+1}{r-i-1} Q_{1,i+1}(t)- \frac{2(r-i+1)}{r-i} t Q_{1,i}(t) + \frac{r-i+2}{i}t^2 Q_{1,i-1}(t)\right )\right]_{t^{i+1}} 
$$
for $p+1 \leq i \leq g$. We obtain the first assertion of $(2)$ by a straightforward calculation. The remaining assertion can be easily handled. 
\end{proof}

The corollary immediately implies Theorem \ref{thm:main3} $(1)$ and $(2)$. For the remaining assertion $(3)$, we assume that $d=\deg L \gg 0$ and $k \leq  p \leq g+1$ unless otherwise stated. Recall from \cite[Theorem C]{EL15} that $\kappa_{r-p,1}(C,L)$ is a polynomial in $d$. We may write
$$
\kappa_{r-p,1}(C,L) = \frac{b_{p}}{a_{p}!}d^{a_{p}} + O(d^{a_{p}-1})~\text{ and }~\overline{\kappa}_{r-p,1} = \frac{b_{p} p!}{a_{p}!} d^{a_{p}-p-1} + O(d^{a_{p}-p-2}),
$$
where
$$
a_{p}:=\dim C_{p}^1~~\text{ and }~~b_{p}:=[C_{p}^1] \cdot  x_{p}^{a_{p}}. 
$$
We set
$$
\ell_{p}:=\begin{cases}
 \frac{b_{p} p!}{a_{p}!} & \text{for $k \leq p \leq g+1$}\\
 0 & \text{otherwise}.
\end{cases}
$$
We recall from \cite[Theorem 1]{FHL84} that $a_{p+1}-a_p = 1$ or $2$. We put $\overline{a}_p:=a_p - p$. Then $\overline{a}_{p+1} - \overline{a}_p = 0$ or $1$. 
By Proposition \ref{prop:compBScoeff}, we can write
$$
c_i=(\ell_{i+1}d^{\overline{a}_{i+1}} -2 \ell_id^{\overline{a}_i} + \ell_{i-1}d^{\overline{a}_{i-1}}) + O(d^{\overline{a}_{i+1}-1})~~\text{ for $k-1 \leq i \leq g$}.
$$
Thus
$$
c_{i+1} - c_i = (\ell_{i+2}d^{\overline{a}_{i+2}}  -3 \ell_{i+1}d^{\overline{a}_{i+1}} + 3\ell_id^{\overline{a}_i} - \ell_{i-1}d^{\overline{a}_{i-1}}) + O(d^{\overline{a}_{i+2}-1})~~\text{ for $k-1 \leq i \leq g-1$}.
$$
If $\overline{a}_{i+2} > \overline{a}_{i+1}$ (i.e., $a_{i+2}-a_{i+1}=2$), then $c_{i+1}>c_i$ since $\ell_{i+2} > 0$ and $d \gg 0$. However, if $\overline{a}_{i+2} = \overline{a}_{i+1}$, then we need to understand the relation between $\ell_{i+2}$ and $\ell_{i+1}$. 

\medskip

Now, we denote by $\maxirr(C_{p}^1)$ the set of maximal irreducible components of $C_{p}^1$. Here a maximal irreducible component means a maximal-dimensional irreducible component, so  $\dim P = a_p$ for all $P \in \maxirr(C_{p}^1)$. We define
$$
b_{p}(P):=[P] \cdot x_{p}^{a_{p}}~~\text{ and }~~\ell_{p}(P):=\frac{b_{p}(P)p!}{a_{p}!}~~\text{ for $P \in \maxirr(C_{p}^1)$}.
$$
Although the scheme structure on $P$ may not be uniquely determined, the intersection number $[P] \cdot x_{p}^{a_{p}}$ is independent of this choice. Note that $b_p(P)=\operatorname{length}(\sO_{C_p^1, \eta_P}) \cdot [P_{\text{red}}] \cdot x_p^{a_p}$, where $\eta_P$ is the generic point of $P$ and $P_{\text{red}}$ is $P$ endowed with the reduced scheme structure. 
Then
$$
b_{p}=\sum_{P \in \maxirr(C_{p}^1)} b_{p}(P)~~\text{ and }~~\ell_{p}=\sum_{P \in \maxirr(C_{p}^1)} \ell_{p}(P).
$$
To understand how the $\ell_i$ are related, we first study the geometry of maximal irreducible components of $C_{p}^1$.

\begin{lemma}\label{lem:sepacompofC_m^1}
Let $Z_1, Z_2 \subseteq C_{p}^1$ be closed subsets. Then $Z_1 \subseteq Z_2$ if and only if $\sigma_{p,1}(Z_1 \times C) \subseteq \sigma_{p,1}(Z_2 \times C)$.
\end{lemma}

\begin{proof}
We only need to show that if $\sigma_{p,1}(Z_1 \times C) \subseteq \sigma_{p,1}(Z_2 \times C)$, then $Z_1 \subseteq Z_2$. To derive a contradiction, suppose that $Z_1 \not\subseteq Z_2$ while $\sigma_{p,1}(Z_1 \times C) \subseteq \sigma_{p,1}(Z_2 \times C)$. We can find $D \in Z_1$ with $D \not\in Z_2$. For every $x \in C$, there are $E_x \in Z_2$ and $y_x \in C$ such that $D+x = E_x + y_x$. Since $x \neq y_x$, it follows that $y_x \in \Supp(D)$. We have $D+x-y_x \in Z_2$. This means that 
$$
C = \bigcup_{y \in \Supp(D)} S_y,
$$
where 
$$
S_y:=\{ x \in C \mid D +x-y \in Z_2\} = \psi_y^{-1}(Z_2) \subseteq C
$$
and $\psi_y \colon C \to C_p$ is a morphism given by $x \mapsto D +x-y$. As $y \not\in S_y$, we see that $S_y \subseteq C$ is a proper closed subset, so we get a contradiction. 
\end{proof}

For an effective divisor $D$ on $C$, we may uniquely write
$$
\lvert D \rvert = \lvert M_D \rvert + F_D,
$$
where $M_D$ is the moving part of the complete linear system $\lvert D \rvert$ and $F_D$ is the fixed part. In other words, $F_D$ is the maximal effective divisor on $C$ such that $h^0(C, D-F_D) = h^0(C, D)$. 

\begin{lemma}\label{lem:constdegF_D}
For a maximal irreducible component $P \subseteq C_{p+1}^1$ with $p \leq g$, there exists a nonempty open subset $U \subseteq P$ such that $\deg F_D$ is constant for all $D \in U$. 
\end{lemma}

\begin{proof}
Notice that 
$$
U_0:=\{ D \in P \mid h^0(C, D)=2\} \subseteq P
$$
is a nonempty open subset by \cite[Lemma IV.1.7]{ACGH85}. 
It is enough to show that
$$
F_n:=\{ D \in U_0 \mid \deg F_D \geq n\} \subseteq U_0
$$
is closed for every integer $n \geq 0$. Note that $F_n = \emptyset$ for $n \geq p+1$ and $F_0 = U_0$. For $1 \leq n \leq p$, we claim that 
$$
F_n = U_0 \cap \sigma_{p+1-n, n}(C_{p+1-n}^1 \times C_n). 
$$
As the addition map $\sigma_{p+1-n, n} \colon C_{p+1-n} \times C_n \to C_{p+1}$ is proper, the claim implies that $F_n \subseteq U_0$ is closed. For $D \in F_n$, we can choose an effective divisor $E \leq F_D$ such that $\deg E = n$. Then $h^0(C, D-E) = h^0(C, D)=2$, so $D = \sigma_{p+1-n,n}(D-E, E)$. Conversely, for $D=\sigma_{p+1-n,n}(D-E, E) \in U_0$ with $D-E \in C_{p+1-n}^1$ and $E \in C_n$, we have $2=h^0(C, D-E) \leq h^0(C, D)=2$, so $h^0(C, D-E)=h^0(C, D)$. This means that $E \leq F_D$ so that $\deg F_D \geq n$. Thus $D \in F_n$. 
\end{proof}

For an irreducible component $Y \subseteq C_m^1$ with $k \leq m \leq p$, let
$$
Y[p]:=\sigma_{m, p-m}(Y \times C_{p-m}) \subseteq C_{p}^1
$$
denote the set-theoretic image. Note that $Y[m]=Y$ and $Y[p]=\sigma_{p-1,1}(Y[p-1] \times C)$ for $p \geq m+1$. Since $\dim Y = \dim Y[p] - (p-m)$ and $\dim C_m^1 \leq \dim C_{p}^1 - (p-m)$, it follows that if $Y[p]$ is a maximal irreducible component of $C_{p}^1$, then $Y$ is also a maximal irreducible component of $C_m^1$ (more generally, $Y[m']$ is a maximal irreducible component of $C_{m'}^1$ for every $m \leq m' \leq p$). We say that a maximal irreducible component $P \subseteq C_{p}^1$ is \emph{primitive} if there is no maximal irreducible component $P' \subseteq C_{p-1}^1$ such that $P=P'[p]$. Notice that $P \subseteq C_{p}^1$ is not a primitive maximal irreducible component if and only if $P \subseteq \sigma_{p-1,1}(C_{p-1}^1 \times C)$. For any maximal irreducible component $P \subseteq C_{p}^1$, there is a primitive component $P' \subseteq C_m^1$ for some $k \leq m \leq p$ such that $P=P'[p]$. We say that $P'$ is a \emph{primitive ancestor} of $P$. By Lemma \ref{lem:sepacompofC_m^1}, a primitive ancestor is uniquely determined.

\begin{proposition}\label{prop:primcomp}
Let $P \subseteq C_{p}^1$ be a maximal irreducible component, and $P' \subseteq C_m^1$ be its primitive ancestor with $k \leq m \leq p$. 
\begin{enumerate}
    \item If $P$ is primitive and $D \in P$ is general, then $D$ is base point free.
    \item For each $m \leq m' \leq p$, we endow $P'[m']$ with reduced scheme structure. Then the restricted addition morphism $\sigma_{m', p-m'} \colon P'[m'] \times C_{p-m'} \to P$ has degree ${p-m \choose p-m'}$. In particular, $\sigma_{m, p-m} \colon P' \times C_{p-m} \to P$ is birational.
\end{enumerate}
\end{proposition}

\begin{proof}
By \cite[Lemma IV.1.7]{ACGH85} and Lemma \ref{lem:constdegF_D}, there is a nonempty open subset $U \subseteq P$ such that $h^0(C, D)=h^0(C, M_D)=2$ and $\deg F_D$ is constant for all $D \in U$.

\medskip

\noindent $(1)$ For any $D \in U$, it suffices to show that $\deg F_D=0$. To derive a contradiction, suppose that $\deg F_D \geq 1$. Then $D=\sigma_{p-1,1}(D-x_D, x_D)$ for any $x_D \in \Supp(F_D)$. But $h^0(C, D-x_D) = 2$, so $P \subseteq \sigma_{p-1,1}(C_{p-1}^1 \times C)$. Hence $P$ is not primitive. 

\medskip

\noindent $(2)$ For any $D \in U$, there are $M \in P'$ and $F \in C_{p-m}$ such that $D=\sigma_{m,p-m}(M, F)$, i.e., $D=M+F$. Note that $h^0(C, D)=h^0(C, M)=2$. Now, we further assume that $D \in U$ is general. Then we may assume that $M$ is base point free by $(1)$. Thus $M=M_D$, and hence, $F=F_D$. Notice that $\deg M_D = m$ and $\deg F_D = p-m$. This means that if $D=\sigma_{m',p-m'}(M', F')$ for $M' \in P'[m'], F' \in C_{p-m'}$, then $M_D \leq M'$ and $F' \leq F_D$. Thus $\sigma_{m', p-m'} \colon P'[m'] \times C_{p-m'} \to P$ has degree ${p-m \choose p-m'}$.
\end{proof}

\begin{lemma}\label{lem:ell(P)}
Let $P \subseteq C_{p}^1$ be a maximal irreducible component, and $P' \subseteq C_m^1$ be its primitive ancestor with $k \leq m \leq p$. We assume that $P'[m']$ is generically reduced for every $m \leq m' \leq p$. Then the following hold:
\begin{enumerate}
    \item We have 
    $$
    \ell_{p}(P)={p \choose m}\ell_{m}(P').
    $$
    \item Assume that $m \geq 3$. Let $\ell_{m'}'(P'):=\ell_{m'}(P'[m'])$ for $m \leq m' \leq p$ and $\ell_{m'}'(P'):=0$ for $m' \leq m-1$. Then
    $$
    \ell_{p}(P) - 3\ell_{p-1}'(P') + 3\ell_{p-2}'(P') - \ell_{p-3}'(P') = {p-3 \choose m-3} \ell_m(P').
    $$
\end{enumerate}
\end{lemma}

\begin{proof}
$(1)$ By Proposition \ref{prop:primcomp} $(2)$, the restricted addition morphism $\sigma_{m, p-m} \colon P' \times C_{p-m} \to P$ is birational. Note that $a_{p} = a_m + (p-m)$. Note that
$$
b_{p}(P)=[P] \cdot x_{p}^{a_{p}} = [P' \times C_{p-m}] \cdot (\pr_1^* x_m + \pr_2^* x_{p-m})^{a_{p}} ={a_{p} \choose a_m} [P'] \cdot x_m^{a_m} = {a _{p}\choose a_m} b_m(P'),
$$
where $\pr_1 \colon C_m \times C_{p-m} \to C_m$ and $\pr_2 \colon C_m \times C_{p-m} \to C_{p-m}$ are the projections. Then
$$
\ell_{p}(P)={a_{p} \choose a_m} \frac{b_m(P') p!}{a_{p}!} = 
{p \choose m} \frac{b_m(P') m!}{a_m!} = {p \choose m} \ell_m(P').
$$

\medskip

\noindent $(2)$ By $(1)$, we have
$$
\ell_{p}(P) - 3\ell_{p-1}'(P') + 3\ell_{p-2}'(P') - \ell_{p-3}'(P') 
= \left( {p \choose m} - 3 {p-1 \choose m} + 3{p-2 \choose m} - {p-3 \choose m} \right) \ell_m(P').
$$
Successively applying Pascal identity, we obtain the result. 
\end{proof}

\begin{remark}
In the situation of Lemma \ref{lem:ell(P)} without the generic reduced assumption, one can show that $C_p^1$ is generically reduced along $P$ if and only if $C_m^1$ is generically reduced along $P'$. Consequently, $C_{m'}^1$ is generically reduced along $P'[m']$ for every $m \leq m' \leq p$. For a general $D \in P$ with $D=\sigma_{m, p-m}(M_D, F_D)$, we may assume that $F_D$ is general by Proposition \ref{prop:primcomp} $(2)$. By \cite[Lemma IV.1.5]{ACGH85}, 
$$
\dim T_D P = 2p-1-g+k_D~~\text{ and }~~\dim T_{M_D} P' = 2m-1-g+k_{M_D}, 
$$
where 
\begin{align*}
& k_{M_D}:= \dim \ker( H^0(C, M_D) \otimes H^0(C, K_C-M_D) \longrightarrow H^0(C, K_C)) = h^0(C, K_C-2M_D )\\
& k_D:=\dim \ker( H^0(C, D) \otimes H^0(C, K_C-D) \longrightarrow H^0(C, K_C)) = h^0(C, K_C-2M_D - F_D).
\end{align*}
Since $k_D=k_{M_D} - (p-m)$, it follows that $\dim T_D P -  (p-m) = \dim T_{M_D} P'$.
On the other hand, recall that $\dim P = \dim P' + (p-m)$.
Hence $\dim P = \dim T_D P$ if and only if $\dim P' = \dim T_{M_D} P'$. 
\end{remark}

Now, we complete the proof of Theorem \ref{thm:main3}.

\begin{proof}[Proof of Theorem \ref{thm:main3}]
As noted before, the assertions $(1)$ and $(2)$ immediately follow from Corollary \ref{cor:compBScoeff}. For the assertion $(3)$, assume that $C$ is not hyperelliptic and every maximal component of $W_{p}^1(C)$ is generically reduced for each $k \leq p \leq g+1$. As the restricted Abel--Jacobi morphism $u_{p} \colon C_{p}^1 \to W_{p}^1(C)$ is a generically $\nP^1$-fibration, $C_{p}^1$ is generically reduced along every maximal irreducible component. Fix an integer $i$ with $k-1 \leq i \leq g-2$. For a maximal irreducible component $P \subseteq C_{i+2}^1$, denote by $P' \subseteq C_m^1$ its primitive ancestor. Then 
$$
c_{i+1} - c_i = \sum_{P \in \maxirr(C_{i+2}^1)} (\ell_{i+2}(P) - 3\ell_{i+1}'(P')+3\ell_i'(P')-\ell_{i-1}'(P')) d^{\overline{a}_{i+2}} + O(d^{\overline{a}_{i+2}-1}),
$$
so $c_{i+1}-c_i>0$ by Lemma \ref{lem:ell(P)} since $m \geq 3$ and $\ell_m(P') > 0$. 
\end{proof}

\begin{remark}\label{rem:HBN}
Assume that $C$ is a general curve of genus $g$ and gonality $k$ with $3 \leq k \leq \lfloor (g+1)/2 \rfloor$. By \cite[Theorem 1.2]{Larson21} and \cite[Theorem 1.2]{LLV25}, we have
$$
W_m^1(C)=\begin{cases} 
\Sigma_{1,1}(m) & \text{for $k \leq m < (g+2)/2$} \\
\Sigma_{1,1}(m) \cup \Sigma_{1,0}(m) & \text{for $(g+2)/2 \leq m \leq g-k+2$}\\
\Sigma_{1,0}(m) & \text{for $g-k+2<m\leq g+1$},
\end{cases}
$$
where $\Sigma_{1,1}(m)$ and $\Sigma_{1,0}(m)$ are irreducible and generically reduced unless they consists of reduced points. Note that 
$$
\dim \Sigma_{1,1}(m)=m-k~~\text{ and }~~\dim \Sigma_{1,0}(m)=2m-g-2.
$$
Thus $W_m^1(C)$ has exactly one maximal irreducible component except when $m=g-k+2$.
\end{remark}

\begin{remark}
Assume that $C$ is not hyperelliptic.
A maximal irreducible component of $C_{p}^1$ may be generically non-reduced in general (see e.g., \cite[Theorem 5.9]{Coppens88}). Lemma \ref{lem:ell(P)} remains valid without the assumption that $P$ and $P'$ are generically reduced provided that $\operatorname{length} (\sO_{C_{p}^1, \eta_P}) = \operatorname{length} (\sO_{C_m^1, \eta_{P'}})$, where $\eta_P$ and $\eta_{P'}$ are generic points of $P$ and $P'$, respectively (this equality implies that $\operatorname{length} (\sO_{C_{m}^1, \eta_{P'}}) = \operatorname{length} (\sO_{C_{m'}^1, \eta_{P'[m']}})$ for all $m \leq m' \leq p$). If this equality holds for all non-reduced maximal irreducible components of all $C_{p}^1$, then Theorem \ref{thm:main3} $(3)$ can be shown similarly. However, we do not know whether this equality always holds. 
\end{remark}

\bibliographystyle{alpha}

\begin{thebibliography}{[39]}
	
	\bibitem{Agostini24}
	Daniele Agostini, {\em The Martens--Mumford theorem and the Green--Lazarsfeld secant conjecture}, J. Algebraic Geom. \textbf{33} (2024), 629--654.

    \bibitem{AN10}
	Marian Aprodu and Jan Nagel, \textit{Koszul cohomology and algebraic geometry}, University Lecture Series, \textbf{52} (2010), Amer. Math. Soc., Providence, RI.

    \bibitem{AC81}
    Enrico Arbarello and Maurizio Cornalba, \textit{Footnotes to a paper of Beniamino Segre}, Math. Ann. \textbf{256} (1981), 341--362.

    \bibitem{ACGH85}
    Enrico Arbarello, Maurizio Cornalba, Joe Harris, and Philip A. Griffith, \textit{Geometry of algebraic curves, vol. I}, Grundlehren Math. Wiss. \textbf{267} (1985), Springer-Verlag, Berlin.

    \bibitem{BS95} 
    Mauro C. Beltrametti and Andrew J. Sommese, \textit{The adjunction theory of complex projective varieties}, de Gruyter Expositions in Mathematics \textbf{16} (1995), Walter de Gruyter and Co., Berlin.

    \bibitem{CCDL19}
    Wouter Castryck, Filip Cools, Jeroen Demeyer, and Alexander Lemmens, \textit{Computing graded Betti tables of toric surfaces}, Trans. Amer. Math. Soc. \textbf{372} (2019), 6869--6903.

    \bibitem{Ciliberto83}
    Ciro Ciliberto, \textit{Sul grado dei generatori dell'anello canonico di una superficie di tipo generale}, Rend. Sem. Math. Univ. Politec. Torino \textbf{41} (1983), 83--111.

    \bibitem{CLPS25+}
    Doyoung Choi, Justin Lacini, Jinhyung Park, and John Sheridan, \textit{Singularities and syzygies of secant varieties of smooth projective varieties}, preprint (2025), 	arXiv:2502.19703.

    \bibitem{Coppens88}
    Marc Coppens, \textit{A study of the schemes $W_e^1$ of smooth plane curves}, Proceedings of the First Belgian--Spanish Week on Algebra and Geometry, R.U.C.A. (1988), 29--62.

    \bibitem{EL12} 
	Lawrence Ein and Robert Lazarsfeld, \textit{Asymptotic syzygies of algebraic varieties},  Invent. Math. \textbf{190} (2012), 603--646.
    
	\bibitem{EL15}
	Lawrence Ein and Robert Lazarsfeld, {\em The gonality conjecture on syzygies of algebraic curves of large degree}, Publ. Math. Inst. Hautes \'Etudes Sci. \textbf{122} (2015), 301--313.

    \bibitem{EL26}
	Lawrence Ein and Robert Lazarsfeld, {\em Lecture on the syzygies and geometry of algebraic varieties},  University Lecture Series, \textbf{81} (2026), Amer. Math. Soc., Providence, RI.
    
	\bibitem{ENP20}
	Lawrence Ein, Wenbo Niu, and Jinhyung Park, {\em Singularities and syzygies of secant varieties of nonsingular
	projective curves}, Invent. Math. \textbf{222} (2020), 615--665.

    \bibitem{Eisenbud95}
    David Eisenbud, \emph{Commutative Algebra, with a View Toward Algebraic Geometry}, Graduate Texts in Math. \textbf{150} (1995), Springer-Verlag, Berlin.

    \bibitem{EGHP05}
    David Eisenbud, Mark Green, Klaus Hulek, and Sorin Popescu, \textit{Restricting linear syzygies: algebra and geometry}, Compositio Math. \textbf{141} (2005), 1460--1478.

    \bibitem{EH87}
    David Eisenbud and Joe Harris, \textit{The Kodaira dimension of the moduli space of curves of genus $\geq 23$}, Invent. Math. \textbf{90} (1987), 359--387.

    \bibitem{ES09}
    David Eisenbud and Frank-Olaf Schreyer, \textit{Betti numbers of graded modules and cohomology of vector bundles}, J. Amer. Math. Soc.  \textbf{22(3)} (2009), 859--888.

    \bibitem{Erman15}
	Daniel Erman, \textit{The Betti table of a high-degree curve is asymptotically pure}, Recent advances in algebraic geometry, London Math. Soc. Lecture Note Ser. \textbf{417} (2015), Cambridge Univ. Press, Cambridge, 200--206.

    \bibitem{FK16}
    Gavril Farkas and Michael Kemeny, {\em The generic Green-Lazarsfeld secant conjecture}, Invent. Math. \textbf{203} (2016), 265--301. 
    
	\bibitem{FK19}
	Gavril Farkas and Michael Kemeny, {\em Linear syzygies of curves with prescribed gonality}, Adv. Math. \textbf{356} (2019), article no. 106810.

    \bibitem{FHL84}
    William Fulton, Joe Harris, and Robert Lazarsfeld, \textit{Excess linear series on an algebraic curve}, Proc. Amer. Math. Soc. \textbf{92} (1984), 320--322.

	\bibitem{Green84}
	Mark Green, {\em Koszul cohomology and the geometry of projective varieties}, J. Differential Geom. \textbf{19} (1984), 125--171.
		
	\bibitem{GL86}
	Mark Green and Robert Lazarsfeld, {\em On the projective normality of complete linear series on an algebraic
	curve},  Invent. Math. \textbf{83} (1986), 73--90.
	
	\bibitem{GL88}
	Mark Green and Robert Lazarsfeld, {\em Some results on the syzygies of finite sets and algebraic curves}, Compositio Math. \textbf{67} (1988), 301--314.
    
    \bibitem{Kemeny20}
    Michael Kemeny, {\em Projecting syzygies of curves}, Algebraic Geom. \textbf{7} (2020), 561--580.

    \bibitem{LLV25}
    Eric Larson, Hannah Larson, and Isabel Vogt, \textit{Global Brill--Noether theory over the Hurwitz space}, Geom. Topol. \textbf{29} (2025), 193--257.

    \bibitem{Larson21}
    Hannah Larson, \textit{A refined Brill--Noether theory over Hurwitz spaces}, Invent. Math. \textbf{224} (2021), 767--790.

    \bibitem{Lazarsfeld04}
	Robert Lazarsfeld, \textit{Positivity in algebraic geometry I. Classical Setting: line bundles and linear series}, A Series of Modern Surveys in Math. \textbf{48} (2004), Springer-Verlag, Berlin.

    \bibitem{Lemmens18}
    Alexander Lemmens, \textit{On the $n$-th row of the graded Betti table of an $n$-dimensional toric variety}, J. Algebraic Comb. \textbf{47} (2018), 561--584.

    \bibitem{NP94}
    Uwe Nagel and Yves Pittleloud, {\em On graded Betti numbers and geometrical properties of projective varieties}, Manuscripta Math. \textbf{84} (1994), 291--314. 

    \bibitem{NP24+}
    Wenbo Niu and Jinhyung Park, \textit{Effective gonality theorem on weight-one syzygies of algebraic curves}, to appear in Duke Math. J. (cf. arXiv:2405.13446).

    \bibitem{OP01} 
	Giorgio Ottaviani and Raffaella Paoletti, \textit{Syzygies of Veronese embeddings},  Compositio Math. \textbf{125} (2001), 31--37.

    \bibitem{Park22}
	Jinhyung Park,  \textit{Asymptotic vanishing of syzygies of algebraic varieties}, Comm. Amer. Math. Soc. \textbf{2} (2022), 133--148.

    \bibitem{Park25}
    Jinhyung Park, \textit{Asymptotic nonvanishing of syzygies of algebraic varieties}, Math. Ann. \textbf{392} (2025), 751--779.
    
	\bibitem{Rathmann16+}
	J\"{u}rgen Rathmann, {\em An effective bound for the gonality conjecture}, preprint (2016), arXiv:1604.06072.
	
	\bibitem{Voisin02}
	Claire Voisin, {\em Green's generic syzygy conjecture for curves of even genus lying on a K3 surface}, J. Eur. Math. Soc. \textbf{4} (2002), 363--404.
	
	\bibitem{Voisin05}
	Claire Voisin, {\em Green's canonical syzygy conjecture for generic curves of odd genus},
	\newblock Compositio Math. \textbf{141} (2005), 1163--1190.
	
\end{thebibliography}

\end{document}